\documentclass[12pt]{amsart}

\usepackage{amssymb}
\usepackage{amsmath,  amscd, amsthm, amsfonts, amstext,
  amsbsy,  mathrsfs, hyperref,  upgreek, mathtools, stmaryrd}

\usepackage[shadow]{todonotes}

\usepackage{enumerate}

\usepackage{graphicx}

\makeatletter
\@namedef{subjclassname@2010}{%
  \textup{2010} Mathematics Subject Classification}
\makeatother

\newtheorem{theorem}{Theorem}[section]
\newtheorem{lemma}[theorem]{Lemma}
\newtheorem{corollary}[theorem]{Corollary}
\newtheorem{proposition}[theorem]{Proposition}
\newtheorem{question}[theorem]{Question}
\theoremstyle{definition}
\newtheorem{claim}{Claim}[theorem]
\newtheorem{definition}[theorem]{Definition}
\theoremstyle{remark}
\newtheorem{remark}[theorem]{Remark}

\numberwithin{equation}{section}

\renewcommand{\L}{\mathsf{L}}
\newcommand{\W}{\mathsf{W}}
\renewcommand{\c}{\mathsf{c}}
\newcommand{\Lip}{\mathsf{Lip}}
\newcommand{\UCont}{\mathsf{UCont}}
\newcommand{\Comp}{\mathsf{Comp}}
\newcommand{\Hyp}{\mathsf{Hyp}}
\newcommand{\Bor}{\mathsf{Bor}}
\newcommand{\F}{\mathcal{F}}
\newcommand{\RR}{\mathbb{R}}
\newcommand{\pre}[2]{{}^{#1} #2}
\newcommand{\pI}{\mathrm{I}}
\newcommand{\pII}{\mathrm{II}}
\newcommand{\D}{\mathsf{D}}

\newcommand{\id}{\operatorname{id}}
\newcommand{\pow}{\mathscr{P}}
\newcommand{\leng}{\operatorname{lh}}

\newcommand{\Deg}{{\sf Deg}}

\newcommand{\AD}{{\sf AD}}
\newcommand{\ZF}{{\sf ZF}}
\newcommand{\AC}{{\sf AC}}
\newcommand{\DC}{{\sf DC}}
\newcommand{\SLO}{{\sf SLO}}
\newcommand{\BP}{{\sf BP}}

\newcommand{\Nbhd}{\boldsymbol{\mathrm{N}}}

\begin{document}


\baselineskip=17pt



\title[Bad Wadge-like reducibilities]{Bad Wadge-like reducibilities \\ on the Baire space}

\author[L. Motto Ros]{Luca Motto Ros}
\address{Albert-Ludwigs-Universit\"at Freiburg \\
Mathematisches Institut -- Abteilung f\"ur Mathematische Logik\\
Eckerstra{\ss}e, 1 \\
 D-79104 Freiburg im Breisgau\\
Germany}
\email{luca.motto.ros@math.uni-freiburg.de}

\date{\today}

\begin{abstract}
We consider various collections of functions from the Baire space \( \pre{\omega}{\omega} \) into itself naturally arising in (effective) descriptive set theory and general topology, including computable (equivalently, recursive) functions, contraction mappings, and functions which are nonexpansive or Lipschitz with respect to suitable complete ultrametrics on \( \pre{\omega}{\omega} \) (compatible with its standard topology).
We analyze the degree-structures induced by such sets of functions when used as reducibility notions between subsets of \( \pre{\omega}{\omega} \), and we show that the resulting hierarchies of degrees are much more complicated than the classical Wadge hierarchy; in particular, they always contain large infinite antichains, and in most cases also infinite descending chains.
\end{abstract}

\subjclass[2010]{Primary 03E15; Secondary 03E60, 03D55, 03D78, 54C10, 54E40}

\keywords{Wadge reducibility, Lipschitz reducibility, computable function, recursive function, contraction mapping, nonexpansive function, Lipschitz function, (ultra)metric Polish space}

\maketitle

\section{Introduction}

Work in \( \ZF + \DC(\RR) \), where \( \DC(\RR) \) is the Axiom of Dependent Choice over the reals, and let \( \pre{\omega}{\omega} \) denote the Baire space of \( \omega \)-sequences of natural numbers (endowed with the product of the discrete topology on \( \omega \)). Given a set of functions \( \F \) from \( \pre{\omega}{\omega} \) into itself and \( A,B \subseteq \pre{\omega}{\omega} \), we set
\[ 
{A \leq_\F B} \iff  \text{there is some } f \in \F \text{ which reduces } A \text{ to } B,
 \] 
where a function \( f \colon \pre{\omega}{\omega} \to \pre{\omega}{\omega} \) \emph{reduces \( A \) to \( B \)} if and only if \( A = f^{-1}(B) \).
When \( A \leq_\F B \) holds, we say that \( A \) is \emph{\( \F \)-reducible} to \( B \). We also set \( {A \equiv_\F B} \iff {A \leq_\F B \leq_\F A} \) and \( {A <_\F B} \iff {A \leq_\F B \wedge B \nleq_\F A} \). If \( \leq_\F \) is a preorder%
\footnote{A binary relation is called \emph{preorder} if it is reflexive and transitive.} 
(which is the case if e.g.\ \( \F \) contains the identity function and is closed under composition), then the relation \( \equiv_\F \) is an equivalence relation, and hence we can consider the \emph{\( \F \)-degree} of a set \( A \subseteq \pre{\omega}{\omega} \) defined by
\[ 
[A]_\F = \{ B \subseteq \pre{\omega}{\omega} \mid B \equiv_\F A \}.
 \] 
A set \( A \subseteq \pre{\omega}{\omega} \) (equivalently, its \( \F \)-degree \( [A]_\F \)) is called \emph{\( \F \)-selfdual} (respectively, \emph{\( \F \)-nonselfdual}) if \( A \leq_\F \pre{\omega}{\omega} \setminus A \) (respectively, \( A \nleq_\F \pre{\omega}{\omega} \setminus A \)).
The collection of all \( \F \)-degrees will be denoted by \( \Deg (\F) \), and for every \( \Gamma \subseteq \pow(\pre{\omega}{\omega}) \) closed under \( \equiv_\F \) we will denote by \(\Deg_{\Gamma}(\F) \)  the collection of all \( \F \)-degrees of sets in \( \Gamma \). The preorder \( \leq_\F \) canonically induces the partial order \( \leq \) on \( \Deg(\F) \) defined by
\[ 
[A]_\F \leq [B]_\F \iff A \leq_\F B,
 \] 
and \( (\Deg(\F), \leq ) \) (sometimes simply denoted by \( \Deg(\F) \) again)  is called \emph{structure of the \( \F \)-degrees} or \emph{degree-structure induced by \( \F \)} or \emph{\( \F \)-hierarchy (of degrees)}. Similar terminology and notation will be used when considering the restriction of \( \leq \) to \( \Deg_\Gamma(\F) \) for some \( \Gamma \) as above.

Several preorders of the form \( \leq_\F \) have been fruitfully considered in the literature until now, including those where \( \F \) is one of the following sets of functions (see Section~\ref{sec:definitions} for the omitted definitions):
\begin{enumerate}[(1)]
\item
the collection \( \L \) of all \emph{nonexpansive functions} and the collection \( \W \) of all \emph{continuous functions} (see e.g.\ the survey paper~\cite{VanWesep:1978} or the more recent~\cite{Andretta:2007});
\item
the collection \( \Lip \) of all \emph{Lipschitz functions} and the collection \( \UCont \) of all \emph{uniformly continuous functions} (see~\cite{MottoRos:2010});
\item
the collection \( \Bor \) of all \emph{Borel(-measurable) functions} (see~\cite{Andretta:2003d});
\item
for \( 1 \leq \alpha < \omega_1 \), the collection \( \D_\alpha \) of all \emph{\( \boldsymbol{\Delta}^0_\alpha \)-functions}, i.e.\ of those \( f \colon \pre{\omega}{\omega} \to \pre{\omega}{\omega} \) such that \( f^{-1}(D) \in \boldsymbol{\Delta}^0_\alpha \) for all \( D \in \boldsymbol{\Delta}^0_\alpha \) (see~\cite{Andretta:2006} for the case \( \alpha = 2 \) and~\cite{MottoRos:2009} for arbitrary \( \alpha \)'s);
\item
for \( 1 \leq \gamma < \omega_1 \) an additively closed ordinal,%
\footnote{The ordinal \(\gamma\) is \emph{additively closed} if \( \alpha + \beta < \gamma \) for every \( \alpha, \beta < \gamma \). This condition is required to ensure that \( \mathscr{B}_{\gamma} \) be closed under composition.}
the collection \( \mathscr{B}_\gamma \) of all functions which are of Baire class \( < \gamma \) (see~\cite{MottoRos:2010});
\item
for \( 1 \leq n \in \omega \), the collection of all \emph{\( \boldsymbol{\Sigma}^1_{2n} \)-measurable functions} (see \cite{MottoRos:2010b}).
\end{enumerate}

Assuming the Axiom of Determinacy \( \AD \), the degree-structure induced by each of the above sets of functions \( \F \) is extremely well-behaved: 
it is well-founded and almost linear, meaning that antichains have size at most \( 2 \) and are in 
fact \emph{\( \F \)-nonselfdual pairs}, i.e.\ they are of the form 
\( \{ [A]_\F, [\pre{\omega}{\omega} \setminus A]_\F \} \) 
for some \( \F \)-nonselfdual \( A \subseteq \pre{\omega}{\omega} \). Therefore all the above notions of reducibility can be reasonably used as tools for measuring 
the complexity of subsets of 
\( \pre{\omega}{\omega} \).

Notice that, except for the case when \( \F \) is the collection of all \( \boldsymbol{\Sigma}^1_{2n} \)-measurable functions, the axiom \( \AD \) is always used only in a local way to determine the degree-structures described above, that is: if 
\( \boldsymbol{\Gamma} \subseteq \pow(\pre{\omega}{\omega}) \) is closed under continuous preimages (i.e.\ \( \boldsymbol{\Gamma} \) is a boldface pointclass), the structure 
\( \Deg_{\boldsymbol{\Gamma}}(\F) = (\Deg_{\boldsymbol{\Gamma}}(\F), \leq) \) of the \( \F \)-degrees of sets in \( \boldsymbol{\Gamma} \) can be fully determined as soon as we 
assume the determinacy of games with payoff set in the closure under complements and finite intersections of \( \boldsymbol{\Gamma} \). Therefore, 
if \( \boldsymbol{\Gamma}\) is the collection of all Borel sets, we do not need to explicitly assume any determinacy axiom because of 
Martin's Borel determinacy (see e.g.\ \cite[Theorem 20.5]{Kechris:1995}).%
\footnote{In fact, in the very special case of Borel sets one can even just work in second-order arithmetic (discharging our original 
assumption \( \ZF + \DC(\RR) \)) by~\cite{Louveau:1988}.} 
Similar considerations will apply to the results of this paper as well, so readers unfamiliar with determinacy axioms may simply drop them, and restrict the attention to Borel subsets of \( \pre{\omega}{\omega} \) throughout the paper.

In general, all classes of functions \( \F \) used as reducibility notions (included the ones above) are required to contain at least all 
nonexpansive functions--- in fact, the condition \( \F \supseteq \L \) is part of the definition of the notion of \emph{set of reductions} introduced in~\cite[Definition 1]{MottoRos:2009}. Why it is so? On the one hand this condition already guarantees that the resulting \( \F \)-hierarchy
 of degrees is well-behaved (in fact, when \( \F \supseteq \L \) only a few characteristics of the induced hierarchy of degrees really depend on the actual \( \F \), 
see Theorem~\ref{th:Fhierarchy} below). On the other hand, the received opinion is that  if \( \F \) lacks such a condition then it is very likely that 
the resulting structure of degrees will not be 
well-behaved, i.e.\ it will contain infinite descending chains and/or infinite antichains (see Definition~\ref{def:verygood}). 
However, besides the trivial example of constant functions briefly considered in~\cite[Section 3]{MottoRos:2009}, to the best 
of our knowledge the problem of whether this opinion is correct has been overlooked in the literature: in particular, 
no ``natural'' example of an \( \F \) inducing an ill-founded hierarchy of degrees (without further set-theoretical assumptions) has 
been presented so far. 

In this note, we fill this gap and confirm the above mentioned intuition by considering 
some concrete examples of sets of functions \( \F \not\supseteq \L \) which naturally 
appear in (effective) descriptive set theory and in general topology. In particular, after recalling some basic notation and results in Section~\ref{sec:definitions},
in Section~\ref{sec:computable} 
we show that when considering the effective counterpart of the
 \( \W \)-hierarchy one gets a considerably complicated structure of degrees (Theorem~\ref{th:computable}). In Section~\ref{sec:contractions} we fully describe the 
hierarchy of degrees induced by the collection \( \c \) of all \emph{contractions} (see Figure~\ref{fig:chierarchy}), showing in particular that such hierarchy contains infinite antichains (Corollary~\ref{cor5}) but no descending chains (Corollary~\ref{cor:6}). The analysis of the \( \c \)-hierarchy  involves 
a characterization of the selfcontractible subsets of \( \pre{\omega}{\omega} \) (see Definition~\ref{defcontractible} and Corollary~\ref{cor2}) which may be of independent 
interest. Finally, 
in Section~\ref{sec:changingmetric} we show that the behavior of the classical \( \L \)-hierarchy heavily relies on the chosen metric: replacing in the definition of \( \L \) (or of \( \Lip \)) the ``standard'' metric \( d \) with another complete ultrametric (still compatible 
with the topology of \( \pre{\omega}{\omega} \)) may in fact yield to extremely wild hierarchies of degrees (Theorems~\ref{th:maind_0} and~\ref{th:maind_1}).

\section{Definitions and preliminaries} \label{sec:definitions}

\subsection*{Basic notation}

The power set of \( X \) is denoted by \( \pow(X) \). The identity function on \( X \) is denoted by \( \id_X \), with the reference to \( X \) dropped when this is not a source of confusion.
When \( A \subseteq X \), we will write \( \neg A \) for \( X \setminus A \) whenever the space \( X \) is clear from the context. The reals are denoted by \( \RR \), and we set \( \RR^+ = \{ r \in \RR \mid r \geq 0 \} \).
The set of natural numbers is denoted by \(\omega\), and \( \pre{\omega}{\omega} \) and 
\( \pre{< \omega}{\omega} \) denote the collections of, respectively, all \(\omega\)-sequences and all finite sequences
 of natural numbers. For \( s \in \pre{< \omega}{\omega} \), \( \leng(s) \) denotes the length of \( s \), and if 
\( x \in \pre{< \omega}{\omega} \cup \pre{\omega}{\omega}\) then \( s {}^\smallfrown{}  x \) denotes 
the concatenation of \( s \) with \( x \). To simplify the notation, when \( s = \langle n \rangle \) for some \( n \in \omega \) we will 
write e.g.\ \( n {}^\smallfrown{} x \) in place of the formally more correct 
\( \langle n \rangle {}^\smallfrown{} x \) (similar simplifications will be applied also to the other notation below). 
If \( A \subseteq \pre{\omega}{\omega} \) and \( s \in \pre{< \omega}{\omega} \), we set 
\( s {}^\smallfrown{}  A = \{ s {}^\smallfrown{} x \mid x \in A  \} \) and
 \( A_{\lfloor s \rfloor} = \{ x \in \pre{\omega}{\omega} \mid s {}^\smallfrown{} x \in A \} \). If 
\( A_n \subseteq \pre{\omega}{\omega} \) for every 
\( n \in \omega \), \( \bigoplus_{n \in \omega} A_n = \bigcup_{n \in \omega} n {}^\smallfrown{}  A_n \). 
When \( A,B \subseteq \pre{\omega}{\omega} \), we also set \( A \oplus B = \bigoplus_{n \in \omega} C_n \), 
where \( C_{2i} = A \) and \( C_{2i+1} = B \) for every \( i \in \omega \).
Given \( n, i  \in \omega \), the symbol \( n^{(i)} \) denotes the unique sequence of length \( i \) which 
is constantly equal to \( n \), and similarly \( \vec{n} \) denotes the \(\omega\)-sequence with constant value \( n \).

\subsection*{The Baire space}

When endowing \( \pre{\omega}{\omega} \) with the product of the discrete topology on \(\omega\), the resulting topological space is called \emph{Baire space}. It is a zero-dimensional Polish space (i.e.\ a completely metrizable second-countable topological space admitting a basis of clopen sets). A compatible complete metric \( d \colon (\pre{\omega}{\omega})^2  \to \RR^+ \) for \(\pre{\omega}{\omega}\) is given by
\[ 
d(x,y) = 
\begin{cases}
0 & \text{if } x = y \\
2^{-n} & \text{if } x \neq y \text{ and $n \in \omega$ is smallest such that } x(n) \neq y(n).
\end{cases}%
 \]
In fact, \( d \) is an \emph{ultrametric} (that is, it satisfies \( d(x,y) \leq \max \{ d(x,z),d(y,z) \} \) for every \( x,y,z \in \pre{\omega}{\omega} \)), and it will be referred to as the \emph{standard metric} on \( \pre{\omega}{\omega} \).

For \( s \in \pre{< \omega}{\omega} \) we set
\[ 
\Nbhd_s = \{ x \in \pre{\omega}{\omega} \mid s \subseteq x \}.
 \] 
The collection \( \{ \Nbhd_s \mid s \in \pre{< \omega}{\omega} \} \) is a countable clopen basis for the topology of \( \pre{\omega}{\omega} \), and in fact it is the collection of all open balls with respect to \( d \).

\subsection*{Classes of functions}

Fix a metric space \( X = (X,d) \). 

\begin{definition} \label{def:definitionfunctions}
A function \( f \colon X \to X \) is called: 
\begin{itemize}
\item
\emph{Lipschitz with constant \( L \in \RR^+ \)}
if \( d(f(x),f(y)) \leq L \cdot d(x,y) \) for every \( x,y \in X \);
\item
\emph{contraction}
if it is Lipschitz with constant \( L < 1 \);
\item
\emph{nonexpansive}
if it is Lipschitz with constant \( L \leq 1 \);
\item
\emph{Lipschitz}
if it is Lipschitz with constant \( L \) for some \( L \in \RR^+ \);
\item
\emph{uniformly continuous}
if for every \( \varepsilon > 0 \) there is \( \delta>0 \) such that \(  d(x,y) < \delta \Rightarrow d(f(x),f(y)) < \varepsilon  \) for every \( x,y \in \pre{\omega}{\omega} \);
\item
\emph{continuous}
if for every \( x \in \pre{\omega}{\omega} \) and every \( \varepsilon > 0 \) there is \( \delta > 0 \) such that \(  d(x,y) < \delta \Rightarrow d(f(x),f(y)) < \varepsilon  \) for every \( y \in \pre{\omega}{\omega} \).
\end{itemize}

The collection of all contractions (respectively, nonexpansive functions, Lipschitz functions, uniformly continuous functions, continuous functions) from the metric space \( \pre{\omega}{\omega} = ( \pre{\omega}{\omega},d) \) into itself will be denoted by \( \c \) (respectively, \( \L \), \( \Lip \), \( \UCont \), \( \W \)).
\end{definition}

When we want to stress the dependence of the
 corresponding definitions on the standard metric \( d \), we will write \( \c(d) \) (respectively, \( \L(d) \), \( \Lip(d) \), \( \UCont(d) \)) in place of \( \c \)
 (respectively, \(\L \), \( \Lip \), \( \UCont \)).%
\footnote{As it is well-known, the collection \( \W \) does not really depend on \( d \) but only on its induced topology.}
 In Section~\ref{sec:changingmetric}, we will consider also different metrics 
\( d' \) on \( \pre{\omega}{\omega} \), and therefore we will denote by \( \c(d') \) (respectively, \( \L(d') \), \( \Lip(d') \),  
\( \UCont(d') \)) the class of all contraction (respectively, nonexpansive, Lipschitz, uniformly continuous) mappings from 
the metric space \( (\pre{\omega}{\omega},d') \) into itself.

\begin{remark} \label{rmk:definitionfunctions}
\begin{enumerate}[i)]
\item
Nonexpansive functions from \( \pre{\omega}{\omega} \) into itself are often called ``Lipschitz functions'' in papers dealing with Wadge theory (see e.g.\ the survey papers~\cite{VanWesep:1978,Andretta:2007}): this is why their collection is usually denoted by \( \L \). However, since in this paper we will also consider the collection \( \Lip \) of all Lipschitz functions (with arbitrary constant), we had to disambiguate the terminology.
\item
The class of all continuous functions from \( \pre{\omega}{\omega} \) into itself is usually denoted by \( \W \) in honor of W.~W.~Wadge, who  initiated a systematic analysis of the associated reducibility preorder \( \leq_ \W \).
\item
Clearly we have
\[ 
\c \subsetneq \L \subsetneq \Lip \subsetneq \UCont \subsetneq \W,
 \] 
and \( \c \) is closed under both left and right composition with nonexpansive functions.
\end{enumerate}
\end{remark}

All classes of functions \( \F \) from \( \pre{\omega}{\omega} \) into itself considered in
 Definition~\ref{def:definitionfunctions} or in the comment following it are closed under composition, and 
(except for \( \c \) and its variants) they contain \( \id = \id_{\pre{\omega}{\omega}} \). Therefore, 
all such \( \F \neq \c \) induce a reducibility \emph{preorder} \( \leq_ \F \), and consequently we can analyze their induced 
degree-structures \( \Deg(\F) = (\Deg(\F),\leq) \). As for \( \F = \c \), the relation \( \leq_{\c} \) is transitive but 
in general not reflexive (see Lemma~\ref{lemma:banach}). Nevertheless it can be naturally extended to a 
preorder, which will be denoted  by \( \leq_{\c} \) again, by setting for  \( A,B \subseteq \pre{\omega}{\omega} \)
\[ 
A \leq_\c B \iff \text{either } A = B \text{ or } A  = f^{-1}( B) \text{ for some } f \in \c.
 \] 
(Equivalently, \( \leq_\c \) is the preorder induced by considering as reducing functions those in the collection \( \F = \c \cup \{ \id \} \): notice that such set of functions remains closed under composition.)  
We will see in Lemma~\ref{lemma:banach} that all sets \( A \subseteq \pre{\omega}{\omega} \) are \( \c \)-nonselfdual, i.e.\ that \( A \nleq_\c \neg A \).
Notice also that for every \( A,A',B,B' \subseteq \pre{\omega}{\omega} \), if \( A = f^{-1}(B) \) for some \( f \in \c \), then
\begin{equation} \label{eq:composition}
 {{A' \leq_\L A} \wedge {B \leq_\L B'}  \Rightarrow {A' \leq_\c B'}}
\end{equation}
by part iii) of Remark~\ref{rmk:definitionfunctions}. 

\subsection*{Boldface pointclasses}

A \emph{boldface pointclass} \( \boldsymbol{\Gamma} \) is a nonempty collection of subsets of 
\( \pre{\omega}{\omega} \) which is closed under continuous preimages, that is \( B \in \boldsymbol{\Gamma} \)
 whenever \( B \leq_\W A \) for some \( A \in \boldsymbol{\Gamma} \). The \emph{dual} of \( \boldsymbol{\Gamma} \) 
is the boldface pointclass \( \check{\boldsymbol{\Gamma}} = \{ \neg A \mid A \in \boldsymbol{\Gamma} \} \), and 
the associated \emph{ambiguous pointclass} is the boldface pointclass 
\( \boldsymbol{\Delta}_{\boldsymbol{\Gamma}} = \boldsymbol{\Gamma} \cap \check{\boldsymbol{\Gamma}} \). A
 boldface pointclass \( \boldsymbol{\Gamma} \) is \emph{nonselfdual} if 
\( \boldsymbol{\Gamma} \neq \check{\boldsymbol{\Gamma}} \), and \emph{selfdual} otherwise. A set
 \( A \subseteq \pre{\omega}{\omega} \) is \emph{properly in \( \boldsymbol{\Gamma} \)} or is a \emph{proper 
\( \boldsymbol{\Gamma} \) set} if \( A \in \boldsymbol{\Gamma} \setminus \check{\boldsymbol{\Gamma}} \). 
Given a boldface pointclass \( \boldsymbol{\Gamma} \) and a collection of functions \( \F \) from 
\( \pre{\omega}{\omega} \) into itself, we say that \( A \subseteq \pre{\omega}{\omega} \) is 
\emph{\( \F \)-complete for \( \boldsymbol{\Gamma} \)} if and only if \( A \in \boldsymbol{\Gamma} \) and 
\( B \leq_\F A \) for every \( B \in \boldsymbol{\Gamma} \). When \( \F \subseteq \W \), \( A \) is \( \F \)-complete for
 \( \boldsymbol{\Gamma} \) if and only if 
\( \boldsymbol{\Gamma} = \{ B \subseteq \pre{\omega}{\omega} \mid B \leq_\F A \} \); moreover, in this case 
\( \boldsymbol{\Gamma} \) is nonselfdual if and only if \( A \) is \( \F \)-nonselfdual. Examples of nonselfdual boldface
 pointclasses are the levels \( \boldsymbol{\Sigma}^0_\xi \) and \( \boldsymbol{\Pi}^0_\xi \) 
(for \( 1 \leq \xi < \omega_1 \)) of the classical stratification of the Borel subsets of \( \pre{\omega}{\omega} \).

\subsection*{Lipschitz games and determinacy axioms}

Given \( A,B \subseteq \pre{\omega}{\omega} \), the so-called \emph{Lipschitz game} \( G_\L(A,B) \) (with payoff sets \( A \) and \(B \)) is the two-player zero-sum infinite 
game in which the two players \( \pI \) and \( \pII \) take turns in playing natural numbers, so that after \(\omega\)-many turns \( \pI \) will have 
enumerated a sequence \( a \in \pre{\omega}{\omega} \) and \( \pII \) will have enumerated a sequence \( b \in \pre{\omega}{\omega} \): the winning 
condition for \( \pII \) is then \( a \in A \iff b \in B \). 

A \emph{strategy for player \( \pI \)} is simply a function \( \sigma \colon \pre{< \omega}{\omega} \to \omega \), and for every
 \( y \in \pre{\omega}{\omega} \) we denote by \( \sigma * y \) the \( \omega \)-sequence enumerated by \( \pI \) in a play of \( G_\L \) in
 which \( \pII \) enumerates \( y \) and \( \pI \) follows \( \sigma \), i.e.\ \( \sigma*y = \langle \sigma(y \restriction n) \mid n \in \omega \rangle \). 
Similarly, a \emph{strategy for \( \pII \)} is a  function \( \tau \colon \pre{< \omega}{\omega} \setminus \{ \emptyset \} \to \omega \), and 
for every \( x \in \pre{\omega}{\omega} \) we set \( x * \tau = \langle \tau(x \restriction n+1) \mid n \in \omega  \rangle \). 

Let \( A,B \subseteq \pre{\omega}{\omega} \): a \emph{winning strategy for \( \pII \)   in the game \( G_\L(A,B) \)} is a strategy \(\tau\) 
for \( \pII \) such that \( x\in A \iff x * \tau \in B \) for every \( x \in \pre{\omega}{\omega} \), and similarly we can 
define \emph{winning strategies for \( \pI \) in \( G_\L(A,B) \)}. Notice that given \(A,B \subseteq \pre{\omega}{\omega} \), at most 
one of  \( \pI \) and \( \pII \) has a winning strategy in \( G_\L(A,B) \). We say that \( G_\L(A,B) \) is \emph{determined} if at least one 
(and, by the comment above, only one) of the players \( \pI \) and \( \pII \) has a winning strategy in \( G_\L(A,B) \).

The next folklore result shows the relationship between the Lipschitz game and the reducibility preorders \( \leq_\c \) and \( \leq_\L \). We fully reprove it here for the reader's convenience.

\begin{proposition}[Folklore] \label{prop:folklore}
Let \( A,B \subseteq \pre{\omega}{\omega} \).
\begin{enumerate}[(1)]
\item
\( A \leq_\c B \iff {{A = B} \vee {\pI \text{ wins } G_\L(\neg B,A)}} \). In fact, if \( \pI \) wins \( G_\L(\neg B,A ) \) then \( A = f^{-1}(B) \) for some \( f \in \c \);
\item
\( A \leq_\L B \iff \pII \text{ wins } G_\L(A,B) \).
\end{enumerate}
\end{proposition}%

\begin{proof} 
(1) Assume first that  \( A \leq_\c B \) and \( A \neq B \), so that there is \( f \in \c \) such that \( A = f^{-1}(B) \).  
Notice that since \( d(f(x),f(y)) \leq \frac{1}{2} d(x,y) \) for all \( x,y \in \pre{\omega}{\omega} \) (because all 
nonzero distances used by \( d \) are of the form \( 2^{-n} \) for some \( n \in \omega \)), for every \( s \in \pre{< \omega}{\omega}\) 
there is a unique \( t_s \in \pre{< \omega}{\omega} \) such that \( \leng(t_s) = \leng(s) + 1 \) and 
\(f(\Nbhd_s) \subseteq \Nbhd_{t_s} \). Define the strategy \(\sigma\) for \( \pI \) by setting
 \( \sigma(s) = t_s(\leng(s)) \) for every  \( s \in \pre{< \omega}{\omega} \): it is easy to check that such a 
strategy is winning in \( G_\L(\neg B,A) \).

Conversely, if \(\sigma\) is a winning strategy for \( \pI \), then the map 
\( f \colon \pre{\omega}{\omega} \to \pre{\omega}{\omega} \mapsto y \mapsto \sigma*y \) is easily 
seen to be a contraction, and \( y \in A \iff f(y) \notin \neg B \iff f(y) \in B \), whence \( A = f^{-1}(B) \).

(2) The proof is similar to that of (1). If \( f \in \L \) witnesses \( A \leq_\L B \) then for every 
\( s \in \pre{< \omega}{\omega} \) there is a unique \( t_s \in \pre{< \omega}{\omega} \) such that \( \leng(t_s) =\leng(s) \) 
and \(f(\Nbhd_s) \subseteq \Nbhd_{t_s} \): setting \( \tau(s) = t_s(\leng(s)-1) \) for each 
\( s \in \pre{< \omega}{\omega} \setminus \{ \emptyset \} \), we get that \( \tau \) is a winning strategy 
for \( \pII \) in \( G_\L(A,B) \). Conversely, if \(\tau\) is a winning strategy for \( \pII \) in \( G_\L(A,B) \), then 
the map \( x \mapsto x * \tau \) is nonexpansive and witnesses \( A \leq_\L  B \).
\end{proof}

Since Lipschitz games can straightforwardly be coded as classical Gale-Stewart games on \(\omega\), the full \( \AD \) implies%
\footnote{Actually, it was conjectured by Solovay that if \( \ZF + \mathsf{V = L(\RR)} \) then \( \AD^\L \Rightarrow \AD \): however, this is still a major open problem in this area.}
the \emph{Axiom of Determincy for Lipschitz games}
\begin{equation}\tag{$\AD^\L$}
\forall A,B \subseteq \pre{\omega}{\omega}\, (G_\L(A,B) \text{ is determined}).
\end{equation}
It immediately follows from Proposition~\ref{prop:folklore} that \( \AD^\L \) is equivalent to the following \emph{Strong Semi-linear Ordering Principle for \( \L \)}
\begin{equation}\tag{$\mathsf{SSLO}^\L$} \label{eq:SLOLstrong}
\forall A,B \subseteq \pre{\omega}{\omega} \, ( {A \leq_\L B } \vee { \neg B \leq_\c A }).
\end{equation}

In particular, since \( \c \subseteq \L \) the principle \( \mathsf{SSLO}^\L \) implies the so-called \emph{Semi-Linear Ordering principle for \( \L \)}, i.e.\ the statement:
\begin{equation} \tag{$\SLO^\L$}
\forall A,B \subseteq \pre{\omega}{\omega} \, (A \leq_\L B \vee \neg B \leq_\L A).
\end{equation}
Actually, by~\cite[Theorem 1]{Andretta:2003} we get that \( \mathsf{SSLO^\L} \) is equivalent to \(  \SLO^\L \) when assuming \(\ZF + \DC(\RR) + \BP\), where \( \BP \) is the statement: ``every \( A \subseteq \pre{\omega}{\omega} \) has the Baire property''.

It is a consequence of \( \SLO^\L \) that if \( \boldsymbol{\Gamma} \) is a nonselfdual boldface pointclass, then every proper \( \boldsymbol{\Gamma} \) set is \( \leq_\L \)-complete for \( \boldsymbol{\Gamma} \) (the converse is always true). Using Proposition~\ref{prop:folklore}, this can be strengthened by replacing nonexpansive functions with contractions.

\begin{corollary}[\( \AD^\L \)] \label{cor4}
Let \( \boldsymbol{\Gamma} \) be a nonselfdual boldface pointclass. Then for every proper \( \boldsymbol{\Gamma} \) set  \( A \) and every \( B \in \boldsymbol{\Gamma} \) we have \( B = f^{-1}(A) \) for some \emph{contraction }\( f\). In particular, a set \( A \subseteq \pre{\omega}{\omega} \) is \( \leq_\c \)-complete for \( \boldsymbol{\Gamma} \) if and only if \( A \in \boldsymbol{\Gamma} \setminus \check{\boldsymbol{\Gamma}} \).
\end{corollary}

\begin{proof}
Let \(A,B \subseteq \pre{\omega}{\omega} \) be distinct sets in \( \boldsymbol{\Gamma} \).
If there is no contraction \( f \) such that \(B =  f^{-1}( A )\),  then \( A \leq_\L \neg B \) by \( \mathsf{SSLO}^\L \) (which is equivalent to \( \AD^\L \)). Since nonexpansive functions are continuous and \( \check{\boldsymbol{\Gamma}} \) is a boldface pointclass, this implies that \( A \) is not a proper \( \boldsymbol{\Gamma} \) set.
\end{proof}

\subsection*{Classical degree-hierarchies}

The \( \L \)-hierarchy and the \( \W \)-hierarchy are the prototype for the degree-structures induced by each of
 the \( \F \)'s mentioned in the introduction. They can be described as follows (for a full proof of Theorem~\ref{th:Lhierarchy}, see e.g.\ \cite{Andretta:2007}).

\begin{theorem}[\( \AD^\L + \BP \)] \label{th:Lhierarchy} 
\begin{enumerate}[(1)]
\item 
\( \leq_\L \) (and hence also the induced partial order on \( \L \)-degrees) is well-founded;
\item 
\( \SLO^\L \) holds, and thus each level of the \( \L \)-hierarchy contains either a single \( \L \)-selfdual degree or an \( \L \)-nonselfdual pair;
\item
at the bottom of the \( \L \)-hierarchy there is the \( \L \)-nonselfdual pair consisting of  \( [\pre{\omega}{\omega}]_\L = \{ \pre{\omega}{\omega} \} \) and \( [\emptyset]_\L = \{ \emptyset \} \);
\item 
successor levels and limit levels of countable cofinality are occupied by a single \( \L \)-selfdual degree;
\item 
at limit levels of uncountable cofinality there is an \( \L \)-nonselfdual pair. 
\end{enumerate}
\end{theorem}

Therefore in the \( \L \)-hierarchy we have an alternation of \( \L \)-nonselfdual pairs with \( \omega_1 \)-blocks of consecutive \( \L \)-selfdual degrees, with \( \L \)-selfdual degrees (followed by an \(\omega_1 \)-block as above) at limit levels of countable cofinality and \( \L \)-nonselfdual pairs at limit levels of uncountable cofinality (see Figure~\ref{fig:Lhierarchy}).

\begin{figure}[!htbp]
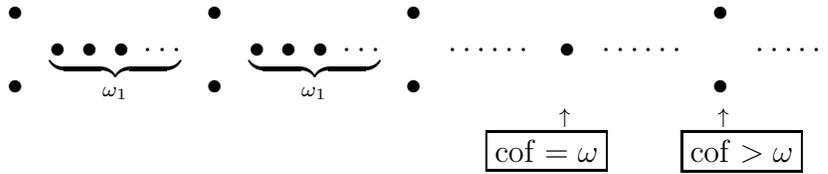

 \centering
\[
\begin{array}{llllllllll}
\bullet & & \bullet & & \bullet & & & &\bullet
\\
& \smash[b]{\underbrace{\bullet \; \bullet \;
\bullet \; \cdots}_{\omega_1}} & &
\smash[b]{\underbrace{\bullet \; \bullet \; \bullet \;
\cdots}_{\omega_1}} & & \cdots\cdots &
\bullet & \cdots\cdots & & \cdots\cdots
\\
\bullet & & \bullet & & \bullet & & & & \bullet
\\
& & & & & &
\stackrel{\uparrow}{\makebox[0pt][r]{\framebox{${\rm cof} = \omega$}}} & &
\, \stackrel{\uparrow}{\makebox[0pt][l]{\framebox{${\rm cof} > \omega$}}}
\end{array}
\]
 \caption{The \( \L \)-hierarchy (bullets represent \( \L \)-degrees).}
 \label{fig:Lhierarchy}
\end{figure}

By the Steel-Van Wesep theorem~\cite[Theorem 3.1]{VanWesep:1978}, under \( \AD^\L + \BP \) one gets that the \( \W \)-hierarchy is obtained 
from the \( \L \)-hierarchy by gluing together each \( \omega_1 \)-block of consecutive \( \L \)-selfdual degrees into a single \( \W \)-selfdual degree, so
 that \( \leq_\W \) is well-founded, at each level of the \( \W \)-hierarchy there is either a single \( \W \)-selfdual degree or a \( \W \)-nonselfdual pair, 
\( \W \)-selfdual degrees coincide exactly with the collapsings of maximal \( \omega_1 \)-blocks of consecutive \( \L \)-selfdual degrees, and 
\( \W \)-nonselfdual pairs coincide exactly with \( \L \)-nonselfdual pairs.
It follows that \( \W \)-nonselfdual pairs alternate with single \( \W \)-selfdual degrees, with \( \W \)-selfdual degrees (followed by a 
\( \W \)-nonselfdual pair) at limit levels of countable cofinality and \( \W \)-nonseldfual pairs at limit levels of uncountable cofinality. The first 
\( \W \)-selfdual degree consists of all nontrivial clopen sets, while the first nontrivial \( \W \)-nonselfdual pair consists of all proper open and proper closed sets
 (Figure~\ref{fig:Whierarchy}).

\begin{figure}[!htbp]
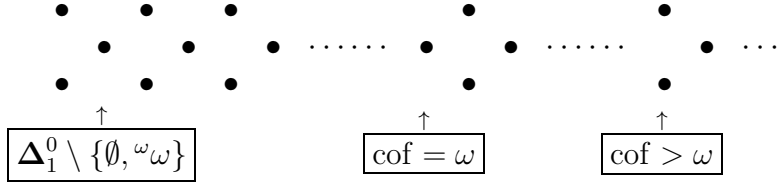

 \centering
\[
\begin{array}{llllllllllllll}
\bullet & & \bullet & & \bullet & & & & \bullet & & & \bullet
\\
& \bullet & & \bullet & & \bullet & \cdots \cdots & \bullet
& & \bullet &\cdots\cdots & & \bullet & \cdots
\\
\bullet & & \bullet & & \bullet & & & & \bullet & & & \bullet
\\
& \, \makebox[0pt]{$\stackrel{\uparrow}{\framebox{$\boldsymbol{\Delta}^0_1 \setminus
\{\emptyset,\pre{\omega}{\omega}\}$}}$}& & & & & &
\, \makebox[0pt]{$\stackrel{\uparrow}{\framebox{${\rm cof} = \omega$}}$}
& & & &
\, \makebox[0pt]{$\stackrel{\uparrow}{\framebox{${\rm cof} > \omega$}}$}
\end{array}
\]
 \caption{The \( \W \)-hierarchy (bullets represent \( \W \)-degrees).}
 \label{fig:Whierarchy}
\end{figure}

Collecting together various easy observations,  
it is possible to show that every collection of functions \( \F \) closed under composition and containing \( \id \) induces a degree-structure very close to the \( \L \)- and the \( \W \)-hierarchy as long as \( \F \supseteq \L \): in fact, the next theorem%
\footnote{In this generality, Theorem~\ref{th:Fhierarchy} first appeared in~\cite{MottoRos:2009}. However, restricted forms of it (in which only the so-called \emph{amenable} collections of functions \( \F \) were considered) already appeared in~\cite{Andretta:2003d,Andretta:2006}.}
 leaves open only the problem of determining what happens after an \( \F \)-selfdual degree, and what happens at limit levels of uncountable cofinality (for a proof of Theorem~\ref{th:Fhierarchy}, see~\cite[Theorem 3.1]{MottoRos:2009}).

\begin{theorem}[\( \AD^\L + \BP \)] \label{th:Fhierarchy} 
Let \( \F \) be a set of functions from \( \pre{\omega}{\omega} \) into itself which is closed under composition and contains \( \id \).
If \( \F \supseteq \L \), then
\begin{enumerate}[(1)]
\item 
\( \leq_\F \) (and hence also the partial order induced on the \( \F \)-degrees) is  well-founded;
\item 
the \emph{Semi-Linear Ordering principle for \( \F \)}
\begin{equation}\tag{$\SLO^\F$}
\forall A,B \subseteq \pre{\omega}{\omega} \, ({A \leq_\F B } \vee {\neg B \leq_\F A})
\end{equation}
is satisfied, and thus
each level of the \( \F \)-hierarchy contains either a single \( \F \)-selfdual degree or an \( \F \)-nonselfdual pair;
\item
the first level of the \( \F \)-hierarchy is occupied by the \( \F \)-nonselfdual pair consisting of \( [\pre{\omega}{\omega}]_\F = \{ \pre{\omega}{\omega} \} \) and \( [\emptyset]_\F = \{ \emptyset \} \);
\item 
after an \( \F \)-nonselfdual pair and at limit levels of countable cofinality there is always a single \( \F \)-selfdual degree;
\item 
if \( A \subseteq \pre{\omega}{\omega} \) is \( \F \)-nonselfdual, then \( [A]_\F = [A]_\L \) (in particular, \( A \) is also \( \L \)-nonselfdual).
\end{enumerate}
\end{theorem}

As recalled above, all the reducibilities \( \F \) mentioned in the introduction satisfy the condition \( \F \supseteq \L \), and in fact, except for the case of the \( \boldsymbol{\Sigma}^1_{2n} \)-measurable functions considered in~\cite{MottoRos:2010b}, under suitable determinacy assumptions their induced degree-structure is either isomorphic to the \( \L \)-hierarchy or to the \( \W \)-hierarchy. More precisely:
\begin{enumerate}[(1)]
\item
\( (\Deg(\Lip), \leq ) \), \( (\Deg(\UCont),\leq) \), and \( (\Deg(\mathscr{B}_\gamma),\leq) \) (for \( 1 \leq \gamma < \omega_1 \) an additively closed ordinal) are all isomorphic to the \( \L \)-hierarchy (Figure~\ref{fig:Lhierarchy}). Moreover, \( (\Deg(\Lip), \leq ) = (\Deg(\UCont),\leq) \);
\item
\( (\Deg(\Bor),\leq) \) and \( (\Deg(\D_\alpha), \leq) \) (for every \( 1 \leq \alpha < \omega_1 \))  are all isomorphic to the \( \W \)-hierarchy (Figure~\ref{fig:Whierarchy}).
\end{enumerate}

\subsection*{A classification of the \( \F \)-hierarchies of degrees}

In~\cite{MottoRos:2012b} it was proposed a rough classification of \( \F \)-hierarchies according to whether they provide an acceptable measure of ``complexity'' for subsets of \( \pre{\omega}{\omega} \).%
\footnote{The two guiding principles for such a classification are the following: (1) 
the \( \F \)-hierarchy must be at least \emph{well-founded}, so that one can associate a rank function to it which measures how much complicated is a given \( \F \)-degree, and (2)
the shorter are the antichains, the better is the classification given by the \( \F \)-hierarchy (this is because it is arguably preferable to have as less as possible distinct \( \F \)-degrees on each of the levels).}

\begin{definition}\label{def:verygood}
Let \( \F \) be a collection of functions from \( \pre{\omega}{\omega} \) to itself which is closed under composition and contains \( \id \). The structure of the \( \F \)-degrees is called:
\begin{enumerate}[-]
\item
\textbf{\emph{very good}} if it is semi-well-ordered, i.e.\ it is well-founded and \( \SLO^\F \) holds;%
\footnote{Of course when we are interested in the restriction of the \( \F \)-hierarchy to some \( \Gamma \subseteq \pow(\pre{\omega}{\omega}) \), then we just require that \( \SLO^\F \) holds for \( A,B \in \Gamma \).} 
\item
\textbf{\emph{good}} if it is a well-quasi-order, i.e.\ it contains neither infinite descending chains nor infinite antichains;
\item
\textbf{\emph{bad}} if it contains infinite antichains;
\item
\textbf{\emph{very bad}} if it contains both infinite descending chains and infinite antichains.
\end{enumerate}
\end{definition}

According to this classification, under
  \( \AD^\L  + \BP \) all the \( \F \)-hierarchies considered above are thus very good, and in fact by Theorem~\ref{th:Fhierarchy} we get that
 \( \F \supseteq \L \) is a sufficient condition for having that the structure of the \( \F \)-degrees is very good. Albeit this is literally not a necessary condition (see the discussion in Section~\ref{sec:questions}), in Sections~\ref{sec:computable}--\ref{sec:changingmetric} we will show  that in many relevant cases if \( \F \not\supseteq \L \) then one gets a bad degree-structure, or even a very bad one.

\begin{remark}
Bad and very bad hierarchies of degrees have been considered in several papers~\cite{Hertling:1993,Hertling:1996, MottoRos:2012b, Ikegami:2012, Schlicht:2012}.  However, all these examples were obtained by considering Wadge-like reducibilites on topological spaces different from \( \pre{\omega}{\omega} \). To the best of our knowledge, the ones reported in the present paper are the first ``natural'' examples of hierarchies of degrees \emph{defined on the classical Baire space} which can be proven to be (very) bad, without any further set-theoretical assumption beyond our basic theory \( \ZF + \DC(\RR) \).
\end{remark}

\subsection*{Selfcontractible sets}

Fix a metric space \( X = (X,d) \).

\begin{definition} \label{defcontractible}
A set \( A \subseteq X \) is called \emph{selfcontractible} if there is a contraction \( f \colon X \to X \) such that \( f^{-1}(A) = A \).
\end{definition}

Notice that, in particular, if \( A \subseteq \pre{\omega}{\omega} \) is selfcontractible and \( B \in [A]_\L \), then \( B \) is selfcontractible as well by iii) of Remark~\ref{rmk:definitionfunctions}.

\begin{remark} \label{rmk:r-contractible}
In Definition~\ref{defcontractible} we could further require that the Lipschitz constant of \( f \) be bounded by some \( 0 < r < 1 \). 
More precisely, given \( 0 < r < 1 \) we could say that a set \( A \subseteq X \) is \emph{\( r \)-selfcontractible} if \( A = f^{-1}(A) \) for some 
\( f \colon X \to X \) such that \( d(f(x),f(y)) \leq r \cdot d(x,y) \) for every \( x,y \in X \). However, it is easy to check that \( A \subseteq X \) 
is selfcontractible if and only if it is \( r \)-selfcontractible \emph{for some \( 0 < r < 1 \)}, if and only if it is \( r \)-selfcontractible 
\emph{for all \( 0 < r < 1 \)}. (For the nontrivial direction, notice that if \( f \) witnesses that \( A \subseteq X \) is selfcontractible, 
then for every \( 0 < r < 1 \) there is \( n(r) \in \omega \) large enough so that \( f^{n(r)} = \underbrace{f \circ \dotsc \circ f}_{n(r)} \) witnesses 
that \( A \) is \( r \)-selfcontractible.) 
\end{remark}%

By the Banach fixed-point theorem, if \( (X,d ) \) is a nonempty complete metric space and \( f \colon X \to X \) is a contraction, then there is a 
(unique) fixed point \( x_f \in X \) for \( f \). From this classical result and~\eqref{eq:composition}, it easily follows that:

\begin{lemma} \label{lemma:banach}
Let \( X = (X,d) \) be a complete metric space. For every \( A \subseteq X \) there is no contraction \( f \colon X \to X \) such that \( f^{-1}(\neg A) 
= A \).

In particular, all sets  \( A \subseteq \pre{\omega}{\omega} \) are \( \c \)-nonselfdual (that is, \( A \nleq_\c \neg A \)), and if \( A \) is \( \L \)-selfdual then \( A \) is not selfcontractible. 
\end{lemma}

Lemma~\ref{lemma:banach} shows that \( \L \)-nonselfduality is a necessary condition for \( A \subseteq \pre{\omega}{\omega} \) being selfcontractible: 
in Corollary~\ref{cor2} we will obtain a full characterization of selfcontractible subsets of \( \pre{\omega}{\omega} \) 
by showing that such a condition is also sufficient (under suitable determinacy assumptions). 

The next simple observation will not be used for the main results of this paper, but it is maybe an interesting fact to be noticed as it shows that each selfcontractible set can be shrunk to arbitrarily small subsets which maintain the same topological complexity.

\begin{proposition}
Let \( X = (X,d) \) be a complete metric space,
\( A \subseteq X \) be selfcontractible, and \( f \) be a witness of this fact. Then for every open neighborhood \( U \) of \( x_f \) 
there is a contraction \( g \colon X \to X \) such that \( A =  g^{-1}( A \cap U ) \). Moreover, if  \( U \) is clopen and \( A \neq X \), then \( A \equiv_\W A \cap U \).
\end{proposition}

\begin{proof}
It is enough to notice that for every \( \varepsilon > 0 \) there is \( n(\varepsilon) \in \omega \) such that 
the range of \( f^{n(\varepsilon)}   = \underbrace{f \circ \dotsc \circ f}_{n(\varepsilon)} \) has diameter \( < \varepsilon \). 
Since \( x_f \), being the fixed point of \( f \), is always in such range, we get that letting \(\varepsilon\) be such that 
\( B(x_f,\varepsilon) = \{ x \in X \mid d(x_f,x) < \varepsilon \} \subseteq U \),  the range of \( f^{n(\varepsilon)} \) is totally contained in 
\( B(x_f,\varepsilon) \), and hence also in \( U \). Finally, \( f^{n(\varepsilon)} \) is clearly a contraction and \( (f^{n(\varepsilon)})^{-1}(A) = A \) because \( f^{-1}(A) = A \) by assumption.

For the last part, since contractions are continuous functions we just need to show \( A \cap U \leq_\W A \): but it is easy to see that this is witnessed by the continuous function \( (\id_X \restriction U ) \cup f_{\bar{y}} \), where  \( f_{\bar{y}} \) is the constant function with value \( \bar{y} \in \neg A \).
\end{proof}

\section{Computable functions} \label{sec:computable} 

Throughout this section, we assume a certain familiarity with the basic concepts and terminology of recursion theory and effective descriptive set 
theory, in particular with the notions of recursive/recursively enumerable subset of \( \omega \) (and its Cartesian products), and with the 
Kleene pointclasses \( \Sigma^0_n \) and \( \Sigma^1_n \) (for \( n \in \omega \)). A good reference for these topics containing all necessary 
definitions is~\cite{Moschovakis:1980}.

As explained in~\cite[Proposition 2.6]{Kechris:1995}, every continuous function from \( \pre{\omega}{\omega} \) into itself can be represented 
with a monotone and length-increasing \( \varphi \colon \pre{< \omega}{\omega} \to \pre{< \omega}{\omega} \). More precisely: 
\( f \colon \pre{\omega}{\omega} \to \pre{\omega}{\omega} \) is continuous if and only if there is 
\( \varphi \colon \pre{< \omega}{\omega} \to \pre{< \omega}{\omega} \) such that
\begin{enumerate}[(a)]
\item
\( s \subseteq t \Rightarrow \varphi(s) \subseteq \varphi(t) \) for every \( s,t \in \pre{< \omega}{\omega} \);
\item
\( \lim_{n \to \infty} \leng(\varphi(x \restriction n)) = \infty \) for all \( x \in \pre{\omega}{\omega} \);
\item
\( f(x) = \bigcup_{n \in \omega } \varphi(x \restriction n) \) for all \( x \in \pre{\omega}{\omega} \).
\end{enumerate}
When (a)--(c) above are satisfied by some \( \varphi \), we say that \(\varphi\) is an \emph{approximating function for \( f \)}.

If we require that a \( \varphi \) as above be computable, then \( f \) itself may be dubbed computable. To be more precise, let \( \mathfrak{G} \colon \pre{< \omega}{\omega} \to \omega \) be the G\"odel bijection, and call a function \( \varphi \colon \pre{< \omega}{\omega} \to \pre{< \omega}{\omega} \) \emph{computable} if and only if  \( \mathfrak{G} \circ \varphi \circ \mathfrak{G}^{-1} \colon \omega \to \omega \) is computable. Then we may introduce the following definition.

\begin{definition} \label{def:computable}

A function \( f \colon \pre{\omega}{\omega} \to \pre{\omega}{\omega} \) is \emph{computable} if and only if there is a computable approximating function \( \varphi \) for \( f \).
\end{definition}

It is easy to check that Definition~\ref{def:computable} is actually equivalent to the definition of a \emph{recursive} function given
 in~\cite[Section 3D]{Moschovakis:1980}.%
\footnote{A function \( f \colon  \pre{\omega}{\omega} \to \pre{\omega}{\omega} \) is called \emph{recursive} if and only if 
the set \( G^f = \{ (x,n) \in \pre{\omega}{\omega} \times \omega \mid f(x) \in \Nbhd_{\mathfrak{G}^{-1}(n)} \} \) is \( \Sigma^0_1 \), i.e.\ it 
is of the form \( \bigcup_{(l,k) \in A} \Nbhd_{\mathfrak{G}^{-1}(l)} \times \{ k \} \) for some recursively enumerable
 \( A \subseteq \omega \times \omega \).}

Since \( \id = \id_{\pre{\omega}{\omega}} \) is computable and the composition of computable functions is computable, setting 
\( \Comp = \{ f \colon \pre{\omega}{\omega} \to \pre{\omega}{\omega} \mid f \text{ is computable} \} \) we get that \( \leq_\Comp \) is a 
preorder, and thus it induces a degree-structure on \( \pow(\pre{\omega}{\omega}) \). Of course, a meaningful use of such a preorder should be confined to subsets of \( \pre{\omega}{\omega} \) which can be defined in a ``recursive fashion'', 
e.g.\ to the Kleene pointclasses \( \Sigma^0_n \) and \( \Sigma^1_n \). However, we are now going to show that even when restricted to
 \( \Pi^0_1 \) or to \( \Sigma^0_2 \), the preorder \( \leq_\Comp \) yields to a quite complicated hierarchy of degrees. This will provide a first 
example of a widely considered%
\footnote{For example the class of computable functions is used to define the Weihrauch reducibility and its induced lattice of degrees: these notions are central in computable analysis, and allows us to e.g.\ classify the computational content of some classical theorems --- see e.g.~\cite{Marconeetal} and the references contained therein.}
 and reasonably complex class of functions lacking the crucial condition%
\footnote{It is easy to see that \( \Comp \) does not even contain e.g.\ constant functions whose unique value is not recursive (as a function from \(\omega\) into itself).}
 \( \F \supseteq \L \) and whose induced degree-structure is very bad. In particular, Theorem~\ref{th:computable} gives a precise mathematical formulation to the common opinion that the 
effective counterpart of the Wadge hierarchy cannot be used as a tool for 
getting a reasonable classification of subsets of \( \pre{\omega}{\omega} \).

\begin{theorem}\label{th:computable}
\begin{enumerate}
\item
The structure of \emph{recursive} subsets of \(\omega\) under inclusion can be embedded into \( \Deg_{\Pi^0_1}(\Comp) \);
\item
the structure of \emph{recursively enumerable} subsets of \(\omega\) under inclusion can be embedded into \( \Deg_{\Sigma^0_2}(\Comp) \).
\end{enumerate}
\end{theorem}

\begin{proof}
Part (1) is somehow implicit in~\cite[Theorem 9]{Fokina:2010}. By~\cite[Theorem 6]{Fokina:2010}, there exists a uniform sequence \( \langle A_n \mid n \in \omega \rangle \)  of nonempty \( \Pi^0_1 \) sets such that for every \( n \in \omega \) there is no computable (in fact, no hyperarithmetical) function \( g \) such that \( g(A_n) \subseteq \bigcup_{m \neq n} A_m \). In particular, none of the \( A_n \)'s can contain a recursive element \( x \in \pre{\omega}{\omega} \), as otherwise the constant function with value \( x \) would contradict the choice of the \( A_n \)'s. 

Given a recursive \( X \subseteq \omega \), set
\[
\psi_0(X) = \bigcup\nolimits_{n \in X} n {}^\smallfrown{} A_n.
\]
Then \( \psi_0(X) \in \Pi^0_1 \) because for every \( n \in \omega \) both \( \Nbhd_{\langle n \rangle}  = n {}^\smallfrown{} \pre{\omega}{\omega}	\) and \( n {}^\smallfrown{} (\neg A_n) \) are in \( \Sigma^0_1 \), the sequence of the \( A_n \)'s is uniform, and under our assumption \( \omega \setminus X \) is recursively enumerable. We claim that \( \psi_0 \) is the desired embedding. 

Let \( X,Y \subseteq \omega \) be two recursive sets. If \( X \subseteq Y \), then the map \( f \colon  \pre{\omega}{\omega} \to \pre{\omega}{\omega} \) defined by 
\[ 
f(x) = 
\begin{cases}
\vec{0} & \text{if } x(0) \notin X \\
x & \text{otherwise}
\end{cases}%
\]
is computable and clearly reduces \( \psi_0(X) \) to \( \psi_0(Y) \) since \( \vec{0} \), being a recursive point of \( \pre{\omega}{\omega} \), does not belong to \( \psi_0(Y) \). 

Conversely, let \( f \) witness \( \psi_0(X) \leq_\Comp \psi_0(Y) \) 
and assume towards a contradiction that there is \( n \in X \setminus Y \). Then since \( \psi_0(X) \cap \Nbhd_{\langle n \rangle} = n {}^\smallfrown{} A_n \), 
the map \( g \colon \pre{\omega}{\omega} \to \pre{\omega}{\omega} \) 
defined by
\[ 
g(x) = \langle f(n {}^\smallfrown{} x)(k+1) \mid k \in \omega \rangle
 \] 
would be computable and such that \( g(A_n) \subseteq \bigcup_{m \in Y} A_m \subseteq \bigcup_{m \neq n} A_m \), contradicting the choice of the \( A_n \)'s. Therefore \( X \subseteq Y \).

(2) We slightly modify the construction of (1). For every recursively enumerable \( X \subseteq \omega \), set
\[ 
\psi_1(X) = \bigcup\nolimits_{n \in X} \bigcup\nolimits_{k,i \in \omega} n {}^\smallfrown{}  0^{(k)} {}^\smallfrown{} (i+1) {}^\smallfrown{}  A_n.
 \] 
Then \( \psi_1 (X) \) is clearly a \( \Sigma^0_2 \) set, and we claim that it is the desired embedding. 

Let \( X,Y  \subseteq \omega \) be 
recursively enumerable sets, and let \( T_X \) be a Turing machine enumerating \( X \).  If \( X \subseteq Y \), then let 
\( \varphi \colon \pre{< \omega}{\omega} \to \pre{< \omega}{\omega} \) be defined by setting \( 
\varphi(\emptyset)  = \emptyset \), \( \varphi(n {}^\smallfrown{} s)  = 
n {}^\smallfrown{} 0^{(\leng(s))} \) if \( n \) is \emph{not} enumerated by \( T_X \) in \( \leq \leng(s) \)-many steps, and 
\( \varphi(n {}^\smallfrown{}  s) = n {}^\smallfrown{} 0^{(k)}  {}^\smallfrown{}  s \restriction (\leng(s)-k) \) if \( n \) is enumerated 
by \( T_X \) in \( k \)-many steps for some \( k \leq \leng(s) \) (for every \( n \in \omega \) and \( s \in \pre{< \omega}{\omega} \)). Then it is easy to check that \( \varphi \) is computable and it satisfies conditions (a)--(b) above, so that \(\varphi\) is an approximating 
function for the computable function 
\( f \colon \pre{\omega}{\omega} \to \pre{\omega}{\omega} \colon x \mapsto \bigcup_{n \in \omega} \varphi(x \restriction n) \). 
Moreover, \( f(n {}^\smallfrown{} x ) \neq n {}^\smallfrown{}  \vec{0} \) if and only if \( n \in X \) and \( x \neq \vec{0} \), and in such case 
\( f(n {}^\smallfrown{} x) = n {}^\smallfrown{} 0^{(k)}  {}^\smallfrown{} x \) for some \( k \in \omega \). This easily implies that
\( f \) reduces \( \psi_1(X) \) to \( \psi_1(Y) \) (since we assumed \( X \subseteq Y \)). 

Conversely, let \( f \in \Comp \) be a witness of 
\( \psi_1(X) \leq_\Comp \psi_1(Y) \), and assume towards a contradiction that there is \( n \in X \setminus Y \). Let \( \varphi \colon \pre{< \omega}{\omega} \to \pre{< \omega}{\omega} \) be a computable approximating function for \( f \), and let \( T \subseteq \pre{< \omega}{\omega} \) be a computable tree such that \( A_n = [T] \), where \( [T] = \{ x \in \pre{\omega}{\omega} \mid \forall n \in \omega \, (x \restriction n \in T) \} \). Define \( \varphi' \colon \pre{< \omega}{\omega} \to \pre{< \omega}{\omega} \) by setting 
\begin{enumerate}[(i)]
\item
\( \varphi'(s) = \emptyset \) if \( s \in T \) and \( \varphi(n {}^\smallfrown{}  1 {}^\smallfrown{} s) \) is of the form \( m {}^\smallfrown{}  0^{(k)} \) for some \( m,k \in \omega \);
\item
\( \varphi'(s) = t \) if \( s \in T \) and \( \varphi(n {}^\smallfrown{} 1 {}^\smallfrown{} s) \) is of the form \( m {}^\smallfrown{} 0^{(k)} {}^\smallfrown{} (i+1) {}^\smallfrown{} t \) for some \( m,k,i \in \omega \) and \( t \in \pre{< \omega}{\omega} \);
\item
\( \varphi'(s) = \varphi'(s \restriction l) {}^\smallfrown{} 0^{(\leng(s))} \) if \( s \notin T \) and \( l < \leng(s) \) is largest such that \( s \restriction l \in T \).
\end{enumerate}%
The map \( \varphi' \) clearly satisfies condition (a) by the fact that \( T \) is closed under subsequences and that \( \varphi \) satisfies (a) as well. To see that \( \varphi' \) satisfies also condition (b), notice that if \( x \in A_n = [T] \) then \( \varphi(x \restriction l) \) must be of the form \( m {}^\smallfrown{} 0^{(k)} {}^\smallfrown{}  (i+1) {}^\smallfrown{}  t \) for all large enough \( l \in \omega \) because \( n {}^\smallfrown{} 1 {}^\smallfrown{} x \in \psi_1(X) \) and \( \varphi \) is an approximating function for the reduction 
\( f \) of \( \psi_1(X) \) to \( \psi_1(Y) \), while if \( x \notin A_n \) then for all large enough \( l \in \omega \) one has \( x \restriction l \notin T \), and 
hence \( \leng(\varphi'(x \restriction l)) \geq l \). Since \( \varphi' \) is clearly computable, this implies that \( \varphi' \) is an approximating 
function for the computable map 
\( g \colon \pre{\omega}{\omega} \to \pre{\omega}{\omega} \colon x \mapsto \bigcup_{i \in \omega} \varphi'(x \restriction i) \). 
Moreover, by the choice of \( f \) and \( \varphi \) one easily gets that%
\footnote{In fact, the function \( g \) would witness that \( A_n \leq_\Comp \bigcup_{m \in Z } A_m \) for every \( Z \supseteq Y \).} 
\( g(x) \in \bigcup_{m \in Y} A_m \subseteq \bigcup_{m \neq n} A_n \)  for every \( x \in A_n \), contradicting the choice of the \( A_n \)'s. 
Therefore \( X \subseteq Y \), as required.
\end{proof}

Obviously, Theorem~\ref{th:computable} can be relativized to any oracle \( z \in \pre{\omega}{\omega} \). Moreover, using the same methods 
one can easily see that similar results hold  when replacing \( \Comp \) with other larger classes of functions which are defined in an 
``effective way'': for example, one can show that the structure of hyperarithmetical subsets of \(\omega\) under inclusion can be embedded 
into the degree-structure \( \Deg_{\Delta^1_1}(\Hyp) \), where \( \Hyp \) is the collection of all hyperarithmetical functions from 
\( \pre{\omega}{\omega} \) into itself.

\section{Contractions} \label{sec:contractions}

Many of the following results will be stated assuming either \( \AD^\L \) or \( \AD^\L + \BP \). As recalled in Section~\ref{sec:definitions}, both 
these assumptions are
 (seemingly weaker) consequences of \( \AD \), so the reader unfamiliar with these special determinacy axioms may 
safely assume the full \( \AD \) throughout the section. Moreover, we remark that all the mentioned determinacy axioms are always used only in a local way (in the sense explained in the introduction): therefore, the restriction of each 
of the results below to the Borel realm is true \emph{without any further assumption beyond \( \ZF + \DC(\RR) \)} --- this feature will be tacitly 
used various times (see e.g.\ Corollary~\ref{cor3}).

\begin{proposition}[\( \AD^\L \)] \label{propdifferentdegrees}
Let \( A,B \subseteq \pre{\omega}{\omega} \). If \( A \) and \( B \) belong to different \( \L \)-degrees (i.e.\ \( A \not\equiv_\L B \)), then
\[ 
A \leq_\c B \iff A \leq_\L B.
 \] 
\end{proposition}

\begin{proof}
One implication is obvious because \( \c \subseteq \L \). For the other direction, if \( A \leq_\L B \) then \( \neg B \nleq_\L A \), as if \( \neg B \leq_\L A \) then we would get \( \neg B \leq_\L B \), and hence also \( A \leq_\L B \leq_\L \neg B \leq_\L A \) (contradicting our assumption \( A \not\equiv_\L B \)). Therefore \( \pI \) wins \( G_\L(\neg B, A ) \) by Proposition~\ref{prop:folklore}(2) and  \( \AD^\L \), whence \( A \leq_\c B \) by Proposition~\ref{prop:folklore}(1).
\end{proof}

\begin{proposition}[\( \AD^\L \)] \label{propsamedegree}
Let \( A ,B \subseteq \pre{\omega}{\omega} \) be distinct sets such that \( A \equiv_\L B \). Then 
\[
A \leq_\c B \iff A \nleq_\L \neg A.
\]
\end{proposition}

\begin{proof}
For the forward direction, assume towards a contradiction that \( A \leq_\c B \) but \( A \leq_\L \neg A \). Then \( B \equiv_\L A \leq_\L \neg A \) by assumption, and since \( A \neq B \) implies that \( A \leq_\c B \) can be witnessed by a function in \( \c \), we would get \( A \leq_\c \neg A \) by~\eqref{eq:composition}, contradicting Lemma~\ref{lemma:banach}. 

If instead \( A \nleq_\L \neg A \), then \( \pI \) wins \( G_\L(\neg A, A) \) by Proposition~\ref{prop:folklore}(2) and \( \AD^\L \), and therefore \( A \) is selfcontractible by Proposition~\ref{prop:folklore}(1), that is \( A \leq_\c A \) can be witnessed by a function in \( \c \). Since \( A \equiv_\L B \), we get \( A \leq_\c B \) by~\eqref{eq:composition} again.
\end{proof}

\begin{remark} \label{rmk:nodeterminacy}
Notice that we actually did not use any determinacy axiom to show \( A \leq_\c B \Rightarrow A \nleq_\L \neg A \) (for \( A,B \) distinct subsets of \( \pre{\omega}{\omega} \) such that \(A \equiv_\L B \)). This fact will be used later in Corollary~\ref{cor5}.
\end{remark}

Despite their simplicity, Propositions~\ref{propdifferentdegrees} and~\ref{propsamedegree} have many interesting consequences. First of all, they provide a characterization of all selfcontractible subsets of \( \pre{\omega}{\omega} \).

\begin{corollary}[\( \AD^\L \)] \label{cor2}
For every \( A \subseteq \pre{\omega}{\omega} \), \( A \) is selfcontractible if and only if it is \( \L \)-nonselfdual.
\end{corollary}

By the Steel-Van Wesep theorem~\cite[Theorem 3.1]{VanWesep:1978}, under \( \AD^\L + \BP \) we have that \( A \subseteq \pre{\omega}{\omega} \) is \( \L \)-selfdual if and only if it is \( \W \)-selfdual.
Therefore we get also the following variant of Corollary~\ref{cor2}.

\begin{corollary}[\( \AD^\L + \BP \)] \label{cor:2'}
For every \( A \subseteq \pre{\omega}{\omega} \), \( A \) is selfcontractible if and only if \( A \) is \( \W \)-nonselfdual.
\end{corollary}

In particular, all sets lying properly in some level of the Baire stratification of the Borel sets are selfcontractible (in fact, this result can be obtained working in \( \ZF + \DC( \RR ) \) alone by Borel determinacy).

\begin{corollary} \label{cor3}
Every proper \( \boldsymbol{\Sigma}^0_\xi \) or \( \boldsymbol{\Pi}^0_\xi \) subset of \( \pre{\omega}{\omega} \) is selfcontractible.
\end{corollary}

\begin{proof}
Proper \( \boldsymbol{\Sigma}^0_\xi \) (respectively, \( \boldsymbol{\Pi}^0_\xi \)) sets are always \( \L \)-nonselfdual (because both \( \boldsymbol{\Sigma}^0_\xi \) and \( \boldsymbol{\Pi}^0_\xi \) are nonselfdual boldface pointclasses and \( \L \subseteq \W \)).
\end{proof}

Corollary~\ref{cor3} can be clearly extended to arbitrary nonselfdual boldface pointclasses \( \boldsymbol{\Gamma} \) assuming sufficiently strong determinacy axioms. Moreover, by Remark~\ref{rmk:r-contractible} one easily gets that Corollaries~\ref{cor2},~\ref{cor:2'}, and~\ref{cor3} can be restated using \( r \)-contractibility (for an arbitrary \( 0 < r < 1 \)) instead of contractibility.

Concerning the structure of the \( \c \)-degrees, we already observed in Lemma~\ref{lemma:banach} that there are no \( \c \)-selfdual degrees. The next corollary of Proposition~\ref{propsamedegree} shows how the \( \c \)- and the \( \L \)-degree of a set \( A \subseteq \pre{\omega}{\omega} \) are related one to the other with respect to inclusion.

\begin{corollary}[\( \AD^\L \)] \label{cor1}
Let \( A \subseteq \pre{\omega}{\omega} \). If \( A  \) is \( \L \)-nonselfdual then \( [A]_\c = [A]_\L \), while if \( A \) is \( \L \)-selfdual then \( [A]_\c = \{ A \} \subsetneq [A]_\L \).
\end{corollary}

\begin{proof}
The inclusion \( [A]_\c \subseteq [A]_\L \) (for an arbitrary \( A \subseteq \pre{\omega}{\omega} \)) follows from \( \c \subseteq \L \). 

Assume first that \( A \) is \( \L \)-nonselfdual and that \( B \in [A]_\L \) is distinct from \( A \): then \( A \leq_\c B \) by Proposition~\ref{propsamedegree}. Moreover, \( B \) is \( \L \)-nonselfdual as well by \( B \equiv_\L A \), so switching the roles of \( A \) and \( B \) we also get \( B \leq_\c A \), and hence \( B \in [A]_\c \). This shows that \( [A]_\L \subseteq [A]_\c \), and hence \( [A]_\c = [A]_\L \). 

Assume now that \( A \) is \( \L \)-selfdual, and that there is \( B \neq A \) such that \( B \in [A]_\c \). Then \( B \equiv_\L A \) (since \( [A]_\c \subseteq [A]_\L \)), and hence \( A \nleq_\c B \) by Proposition~\ref{propsamedegree} again.
\end{proof}

Proposition~\ref{propsamedegree} can also be used to show (in \( \ZF+\DC(\RR) \) alone) that the degree-structure induced by \( \c \) is bad, as it contains very large antichains.

\begin{corollary} \label{cor5}
The preorder \( \leq_\c \) contains antichains of size \( \pre{\omega}{2} \), i.e.\ there is an injection \( \psi \colon  \pre{\omega}{2} \to \pow(\pre{\omega}{\omega}) \) such that \( \psi(x) \) and \( \psi(y) \) are \( \leq_\c \)-incomparable whenever \( x \neq y \). In fact, a \( \leq_\c \)-antichain of size \( \pre{\omega}{2} \) can be found inside every \( \L \)-selfdual degree.
\end{corollary}

\begin{proof}
Set 
\[ 
\psi(x) =\Nbhd_{\langle 0 \rangle} \cup \bigcup \{ \Nbhd_{\langle n+2 \rangle} \mid x(n) = 1 \} .
\]
Then \( \psi \colon \pre{\omega}{2} \to \pow(\pre{\omega}{\omega}) \) is injective and 
\( \psi(x) \equiv_\L \psi(y) \equiv_\L \Nbhd_{\langle 0 \rangle} \) for every \( x,y \in \pre{\omega}{2} \). 
Since \( \Nbhd_{\langle 0 \rangle} \) is clearly \( \L \)-selfdual, the result follows from (the forward direction of)
Proposition~\ref{propsamedegree} (together with Remark~\ref{rmk:nodeterminacy}). 

The additional part is 
obtained in the same way, using the following general claim.

\begin{claim}
For every \( A \neq \pre{\omega}{\omega}, \emptyset \) there is an injection \( \psi \colon \pre{\omega}{2} \to [A]_\L \).
\end{claim}

\begin{proof}[Proof of the Claim]
For \( x \in \pre{\omega}{2} \), set
\[
\psi(x) = {\bigcup_{n \in \omega} (2n) {}^\smallfrown{}  A_{\lfloor n \rfloor} } \cup {\bigcup \{ \Nbhd_{\langle 2n+1 \rangle} \mid x(n) = 1 \}}.
 \]
Fix \( y_0 \notin A \) and \( y_1 \in A \). 
It is then easy to check that the maps \( f , g \colon \pre{\omega}{\omega} \to \pre{\omega}{\omega} \) defined by
\[ 
f(n {}^\smallfrown{} y) = (2n) {}^\smallfrown{} y \quad \text{ and } \quad g(n {}^\smallfrown{} y) = 
\begin{cases}
i {}^\smallfrown{} y & \text{if } n = 2i \\
y_0 & \text{if } n  = 2i+1 \text{ and } x(i) = 0 \\
y_1 & \text{if } n = 2i+1 \text{ and } x(i) = 1
\end{cases}
 \] 
witness \( A \equiv_\L \psi(x) \) for every \( x \in \pre{\omega}{\omega} \).
\end{proof}
\end{proof}

On the other hand, assuming sufficiently strong  determinacy axioms one can show that the \( \c \)-hierarchy is not very bad, i.e.\ that it is at least well-founded.

\begin{corollary}[\( \AD^\L \)] \label{cor:6}
For every \( A,B \subseteq \pre{\omega}{\omega} \),
\[ 
A <_\c B \iff A <_\L B.
 \] 
In particular, further assuming \( \BP \) we get that the preorder \( \leq_\c \) is well-founded.
\end{corollary}

\begin{proof}
If \( A <_\L B \) then \( A <_\c B \) by Proposition~\ref{propdifferentdegrees}. Conversely, assume \( A <_\c B \), so that, in particular, \( A \neq B \). Then 
\( A \leq_\L B \) by \( \c \subseteq \L \). Assume towards a contradiction that \( A \equiv_\L B \): then \( A \nleq_\L \neg A \) by Proposition~\ref{propsamedegree} and \( A \leq_\c B \), whence \( [A]_\L = [A]_\c \) by Corollary~\ref{cor1}. But then \( B \equiv_\c A \), contradicting our choice of \( A \) and \( B \).

In particular, every infinite strictly \( \leq_\c \)-decreasing chain is also strictly \( \leq_\L \)-decreasing, and therefore by Theorem~\ref{th:Lhierarchy}(1) we get that \( \leq_\c \) is well-founded.
\end{proof}

More generally, combining Corollary~\ref{cor1} with Proposition~\ref{propdifferentdegrees}, we get a full description of the degree-structure induced by \( \c \). In fact, the relation \( \leq_\c \) is simply the refinement of \( \leq_\L \) in which all sets belonging to the same \( \L \)-selfdual degree are made pairwise \( \leq_\c \)-incomparable. Therefore the \( \c \)-hierarchy of degrees is obtained from the \( \L \)-hierarchy by splitting each \( \L \)-selfdual degree into the singletons of its elements. Figure~\ref{fig:chierarchy} summarizes the situation  (compare it with Figure~\ref{fig:Lhierarchy}): bullets represent \( \c \)-degrees, while the boxes around them represent the \( \L \)-degrees they come from.

\begin{figure}[!htbp]
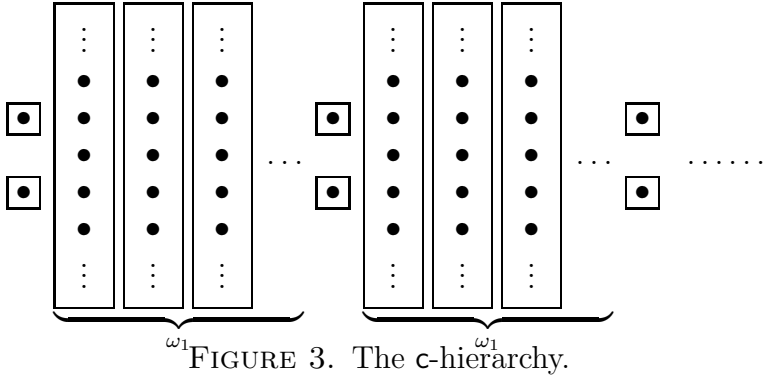

 \centering
\[
\begin{array}{c}
\phantom{\vdots}\\
\phantom{\bullet} \\
\framebox{$\bullet$} \\
\phantom{\bullet}\\
\framebox{$\bullet$} \\
\phantom{\bullet} \\
\phantom{\vdots}
\end{array}%
\smash[b]{
\underbrace{
\framebox{$
\begin{array}{c}
\vdots \\
\bullet \\
\bullet \\
\bullet \\
\bullet \\
\bullet \\
\vdots
\end{array}
$} \;
\framebox{$
\begin{array}{c}
\vdots \\
\bullet\\
\bullet \\
\bullet\\
\bullet\\
\bullet\\
\vdots
\end{array}
$} \;
\framebox{$
\begin{array}{c}
\vdots \\
\bullet\\
\bullet \\
\bullet\\
\bullet\\
\bullet\\
\vdots
\end{array}
$} \;
\dotsc
}_{\omega_1}}
\begin{array}{c}
\phantom{\vdots}\\
\phantom{\bullet} \\
\framebox{$\bullet$} \\
\phantom{\bullet}\\
\framebox{$\bullet$} \\
\phantom{\bullet} \\
\phantom{\vdots}
\end{array}
\smash[b]{
\underbrace{
\framebox{$
\begin{array}{c}
\vdots \\
\bullet \\
\bullet \\
\bullet \\
\bullet \\
\bullet \\
\vdots
\end{array}
$} \;
\framebox{$
\begin{array}{c}
\vdots \\
\bullet\\
\bullet \\
\bullet\\
\bullet\\
\bullet\\
\vdots
\end{array}
$} \;
\framebox{$
\begin{array}{c}
\vdots \\
\bullet\\
\bullet \\
\bullet\\
\bullet\\
\bullet\\
\vdots
\end{array}
$} \;
\dotsc
}_{\omega_1}}
\begin{array}{c}
\phantom{\vdots}\\
\phantom{\bullet} \\
\framebox{$\bullet$} \\
\phantom{\bullet}\\
\framebox{$\bullet$} \\
\phantom{\bullet} \\
\phantom{\vdots}
\end{array}
\;
\dotsc \dotsc 
\]
 \caption{The \( \c \)-hierarchy.}
 \label{fig:chierarchy}
\end{figure}

One may wonder what happens if we further restrict our attention to the collection of all contractions admitting a Lipschitz constant smaller than or equal to a fixed \( 0 < r < 1 \). More precisely, given \( 0 < r < 1 \) let \( \c(r) \) be the collection of all functions \( f \colon  \pre{\omega}{\omega}  \to \pre{\omega}{\omega} \) such that \( d(f(x),f(y)) \leq r \cdot d(x,y) \) for every \( x,y \in \pre{\omega}{\omega} \). Then \( \c(r) \) is closed under composition, and hence (with a little abuse of notation) we can define the preorder
\[ 
A \leq_{\c(r)} B \iff \text{either } A = B \text{ or } A = f^{-1}(B) \text{ for some } f \in \c(r).
 \] 
In particular, \( {\leq_{\c(r)}} = {\leq_\F} \) for \( \F = \c(r) \cup \{ \id \} \).

Given \( 0 <  r < 1 \), let \( n(r) \) be the smallest \( n \in \omega \) such that \( 2^{-(n+1)} \leq r \). Then \( \c(r) = \c(2^{-(n(r)+1)}) \), and the relation \( \leq_{\c(r)} \) admits a characterization via winning strategies for \( \pI \) in suitable reduction games similar to the one we obtained in Proposition~\ref{prop:folklore} for \( \leq_\c \) (which corresponds to the case \( \frac{1}{2} \leq r < 1 \)). In fact, it is enough to replace the Lipschitz game \( G_\L \) with the \emph{\( n(r) \)-Lipschitz game \( G_{n(r) \text{-}\Lip} \)} introduced in~\cite[Section 3]{MottoRos:2011a} to get that:
\begin{enumerate}
\item
\( A \leq_{\c(r)} B \iff A = B \vee {\pI \text{ wins } G_{n(r) \text{-}\Lip}(\neg B ,A)} \). In fact, if \( \pI \) wins \( G_{n(r) \text{-}\Lip}(\neg B ,A) \) then \( A = f^{-1}(B) \) for some \( f \in \c(r) \);
\item
\( A = f^{-1}(B) \) for some Lipschitz function with constant \( 2^{n(r)} \iff \pII \) wins \( G_{n(r) \text{-}\Lip}(A ,B) \).
\end{enumerate}

Using this characterization of \( \leq_{\c(r)} \), one can reprove suitable variants of most of the results needed to determine the corresponding degree-structure \( \Deg(\c(r)) \).%
\footnote{In fact, using the game theoretic characterization mentioned above (together with the fact that \( \Lip \subseteq \W \)), one can also show that we can replace \( \c \) with \( \c(r) \) (for an arbitrary \( 0 < r < 1 \)) in the completeness result Corollary~\ref{cor4}.}
 In particular, the analogue of Proposition~\ref{propsamedegree} in which \( \leq_\c \) is replaced by \( \leq_{\c(r)} \) (for an arbitrary \( 0 < r < 1 \)) is true. (For the forward direction use \( \c(r) \subseteq \c \), while for the backward direction use the fact that under \( \AD^\L \), a set \( A \subseteq \pre{\omega}{\omega} \) is \( \L \)-nonselfdual if and only if it is \( \Lip \)-nonselfdual --- see~\cite{MottoRos:2010}.) Therefore also the analogues of Corollaries~\ref{cor1} and~\ref{cor5} remain true when replacing \( \c \) with \( \c(r) \).

However, not all the results of this section can be generalized to arbitrary preorders of the form \( \leq_{\c(r)} \). For example,
 Proposition~\ref{propdifferentdegrees} fails if \( r < \frac{1}{2} \), because if \( A \subseteq \pre{\omega}{\omega} \) is \( \L \)-selfdual, 
then \( A <_\L 0 {}^\smallfrown{} A \) but \( A \nleq_{\c(r)} 0 {}^\smallfrown{} A \) (in particular, this counterexample shows also that  the first part Corollary~\ref{cor:6} fails for such \( r \)'s as well): nevertheless, we can still prove that \( \leq_{\c(r)} \) is well-founded (under \( \AD^\L + \BP \)) using a slightly different argument.

\begin{corollary}[\( \AD^\L + \BP \)] \label{cor:6'}
Let \( 0 < r < 1 \). Then the preorder \( \leq_{\c(r)} \) is well-founded.
\end{corollary}

\begin{proof}
Assume towards a contradiction that there is a sequence \( \langle A_n \mid n \in \omega \rangle \) of subsets of \( \pre{\omega}{\omega} \) such that \( A_{n+1} <_{\c(r)} A_n \) for every \( n \in \omega \). 

Assume first that there is \( N \in \omega \) such that \( A_n \equiv_\L A_m \) for every \( n,m \geq N \). If \( A_N \) (hence also all the \( A_m \)'s
with \( m \geq N \)) is \( \L \)-nonselfdual, then \( [A_N]_\L = [A_N]_{\c(r)} \) by (the analogue of) Corollary~\ref{cor1}, and hence 
\( A_n \equiv_{\c(r)} A_m \) for every \( n,m \geq N \); if instead \( A_N \) is \( \L \)-selfdual, then \( A_n \nleq_{\c(r)} A_m \) for all distinct 
\( n,m \geq N \) by (the analogue of) Proposition~\ref{propsamedegree}. Thus in both cases we reach a contradiction with our choice of 
the \( A_n  \)'s.

Therefore, passing to a subsequence if necessary, we can assume without loss of generality that \( A_n \not\equiv_\L A_m \) for all distinct \( n,m \in \omega \). But then the sequence of the \( A_n \)'s is also \( \leq_\L \)-descending by \( \c(r) \subseteq \L \), contradicting Theorem~\ref{th:Lhierarchy}(1). 
\end{proof}

This shows that the degree-structure induced by \( \c(r) \) is, under suitable determinacy assumptions, another example of a bad degree-structure which is 
not very bad. However, when \( 0 < r < \frac{1}{2} \) (i.e.\ when \( \c(r) \neq \c \)) the \( \c(r) \)-hierarchy is much more difficult to be described: 
this is mainly due to the counterexample described before Corollary~\ref{cor:6'}. However, using the above game-theoretic characterization of \( \leq_{\c(r)} \) (together with the fact that, under our set-theoretical assumptions, \( \L \)-nonselfduality and \( \Lip \)-nonselfduality coincide) we can 
still give a full description of the \( \leq_{\c(r)} \)-preorders in term of \( \L \)-selfduality and \( \leq_\L \)-reducibility, from which a full description of the \( \c(r) \)-hierarchy can be easily recovered.

\begin{proposition}[\( \AD^\L + \BP \)] 
For every \( 0 < r < 1 \) and every \( A,B \subseteq \pre{\omega}{\omega} \)
\begin{align*}
A \leq_{\c(r)} B  \iff & {A = B} \vee  ({A \nleq_\L \neg A} \wedge {A \leq_\L B} ) \vee \\
&
  ({A \leq_\L \neg A} \wedge {0^{(n(r) + 1)} {}^\smallfrown{} A \leq_\L B}).
\end{align*}
\end{proposition}%

\begin{proof}[Sketch of the proof]
In order to prove the forward direction, assume that \( A \leq_{\c(r)} B \). Since \( {A \leq_{\c(r)} B} \Rightarrow {A \leq_\L B} \) by 
\( \c(r) \subseteq \L \), the unique nontrivial case that needs to be considered is when \( A \leq_\L \neg A \) with \( A \neq B \). Assume towards a contradiction that 
\( {0^{(n(r)+1)} {}^\smallfrown{} A \nleq_\L B} \): then 
\[ 
\neg A \leq_\L A \leq_\L B \leq_\L 0^{(n(r))} {}^\smallfrown{} A \equiv_\L  0^{(n(r))} {}^\smallfrown{} (\neg A), 
\]
so that 
\( B = f^{-1}(\neg A) \) via some Lipschitz function \( f \) with constant \( 2^{n(r)} \). Therefore \( \pII \) would win \( G_{n(r)\text{-}\Lip}(\neg B,A) \), 
and hence \( A \nleq_{\c(r)} B \) because \( \pI \) could not win such a game, a contradiction.

For the backward direction, assume first that \( A \nleq_\L \neg A \) and \( A \neq B \). If \( A \equiv_\L B \), then \(  A \leq_{\c(r)} B \) by 
(the analogue of) Proposition~\ref{propsamedegree}. If instead \( A <_\L B \), then \( A <_\Lip B, \neg B \) as well: hence \( \pII \) cannot win 
\( G_{n(r)\text{-}\Lip}(\neg B , A) \), and since such a game is determined by our assumptions,%
\footnote{Recall that by~\cite[Lemma 6.1]{MottoRos:2011a} the principle \( \AD^\L \) implies that all games of the form \( G_{k\text{-}\Lip}(A,B) \) (for an arbitrary \( k \in \omega \)) are determined.} 
we get \( A \leq_{\c(r)} B \). Finally, assume that 
\( A \leq_\L \neg A \) and \( 0^{(n(r)+1)} {}^\smallfrown{}  A \leq_\L B \) (which in particular implies \( A \neq B \)). If \( A \nleq_{\c(r)} B \), then \( \pI \) could not win
\( G_{n(r)\text{-}\Lip}(\neg B,A) \). Since such a game is determined
 and \( A \leq_\L \neg A \), we would then have that \( B = f^{-1}(A) \) via some Lipschitz function \( f \) with constant 
\( 2^{n(r)} \), which in turn would imply \( B \leq_\L 0^{(n(r))} {}^\smallfrown{} A \), a contradiction.
\end{proof}

\section{Changing the metric} \label{sec:changingmetric}

As long as reducibility preorders \( \leq_\F \) between subsets of \( \pre{\omega}{\omega} \) are concerned,  there are three kinds of sets of 
functions \( \mathcal{F} \)  that have been considered in the literature whose definition actually depends on the standard metric \( d \) on \( \pre{\omega}{\omega} \) (rather than on its
topology), namely:

\begin{enumerate}[(1)]
\item
the collection \( \L = \L(d) \) of nonexpansive functions;
\item
the collection \( \Lip = \Lip(d) \) of all Lipschitz functions (with arbitrary constant);
\item
the collection \( \UCont = \UCont(d) \) of all uniformly continuous functions. 
\end{enumerate}

As recalled in Section~\ref{sec:definitions}, under suitable determinacy assumptions all three degree-structures induced by these notions of reducibility are  very good and isomorphic one to the other (see Figure~\ref{fig:Lhierarchy}); in fact the degree-structures \( (\Deg(\Lip) , \leq) \) and \( (\Deg(\UCont), \leq) \) coincide despite the fact that \( \Lip \subsetneq \UCont \).
A natural question is then the following:

\begin{question}
What happens if we replace \( d \) with another complete (ultra)metric \( d' \) compatible with the topology of \( \pre{\omega}{\omega} \)? Are the degree-structures induced by \( \L(d') \), \( \Lip(d') \), and \( \UCont(d') \) still well-behaved (i.e.\ good or very good)?
\end{question}%

Of course trivial modifications of \( d \), such as replacing the distances \( \langle 2^{-n} \mid n \in \omega \rangle \) used in the definition of \( d \) with any strictly decreasing sequence of reals converging to \( 0 \), yield exactly to the same classes of functions (and hence the same induced degree-structures). However, slightly more elaborated variants can heavily modify the resulting hierarchies of degrees.

\begin{definition} \label{def:d_0}
Let \( d_0  \colon (\pre{\omega}{\omega})^2 \to \RR^+ \) be the metric on \( \pre{\omega}{\omega} \) defined by:
\[ 
d_0(x,y) = 
\begin{cases}
0 & \text{if } x = y \\
d(x,y) & \text{if } x(0) = y(0) \\
\max \{ x(0),y(0) \} & \text{if } x(0) \neq y(0).
\end{cases}%
 \] 
\end{definition} 

\noindent
Thus \( (\pre{\omega}{\omega},d_0) \) is essentially obtained by ``gluing'' together the subspaces \( \Nbhd_{\langle n \rangle} \) of \( (\pre{\omega}{\omega},d) \) by letting all the points in \( \Nbhd_{\langle n \rangle} \) have distance \( \max\{ n,m \} \) from all the points in \( \Nbhd_{\langle m \rangle} \) (for distinct \( n,m \in \omega \)). 

The trivial but crucial observation is that for \( x,y \in \pre{\omega}{\omega} \)
\begin{enumerate}[($\dagger$)]
\item
\( d(x,y) < 1 \iff d_0(x,y) < 1 \), and in such case \( d(x,y) = d_0(x,y) \).
\end{enumerate}
Moreover, \( d(x,y) \leq d_0(x,y) \) for every \( x,y \in \pre{\omega}{\omega} \).

\begin{proposition} \label{prop:d_0compatible}
The metric \( d_0 \) is a complete ultrametric compatible with the product topology on \( \pre{\omega}{\omega} \).
\end{proposition}

\begin{proof} 
Since clearly \( d_0(x,y) = d_0(y,x) \) and \( d_0(x,y) = 0 \iff x=y \) (for all \( x,y \in \pre{\omega}{\omega} \)),
to see that \( d_0 \) is an ultrametric it is enough to fix \( x,y,x \in \pre{\omega}{\omega} \) and show that \( d_0(x,y) \leq \max \{ d_0(x,z),d_0(y,z) \} \): this can be straightforwardly checked by considering various cases, depending on whether the values of \( x(0) \), \( y(0) \), and \( z(0) \) coincide or are distinct. 
Finally, the fact that \( d_0 \) is complete and compatible with the product topology on \( \pre{\omega}{\omega} \) easily follows from 
($\dagger$) above.
\end{proof}

Denote by \( \subseteq^* \) the relation of inclusion modulo finite sets on \( \pow(\omega) \), namely for \( X,Y \subseteq \omega \) set
\[ 
X \subseteq^* Y \iff \exists \bar{k} \in \omega \, \forall k \geq \bar{k} \, (k \in X \Rightarrow k \in Y).
 \]

\begin{theorem} \label{th:maind_0}
Let \( A \subseteq \pre{\omega}{\omega} \) be a \( \W \)-selfdual set.
Then there is a map \( \psi \colon  \pow(\omega) \to [A]_\W \) such that for every \( X,Y \subseteq \omega \) it holds:
\begin{enumerate}[(1)]
\item
if \( X \subseteq^* Y \), then \( \psi(X) \leq_{\L(d_0)} \psi(Y) \);
\item
if \( \psi(X) \leq_{\Lip(d_0)} \psi(Y) \), then \( X \subseteq^* Y \).
\end{enumerate}
\end{theorem}

\begin{proof}
Without loss of generality, we may assume that \( A \leq_{\L(d)} \neg A \) (otherwise we replace \( A \) with \( A \oplus \neg A \)).  Recursively define a sequence \( \langle A_m \mid m \in \omega \rangle \) of subsets of \( \pre{\omega}{\omega} \) by setting:
\begin{align*}
A_0 & =  A, \\
A_{m+1} & =  \bigoplus_{n \in \omega} 0^{(n)} {}^\smallfrown{} A_m.
\end{align*}%
Arguing as in~\cite{MottoRos:2010}, it is easy to check that:
\begin{enumerate}[a)]
\item
\( A_m \leq_{\L(d)} \neg A_m \) (hence, in particular \( A_m \neq \pre{\omega}{\omega} \)) for every \( m \in \omega \);
\item
\( A_n \leq_{\L(d)} A_m \) for every \( n \leq m \in \omega\);
\item
\( A_m \nleq_{\Lip(d)} A_n \) for every \( n < m \in \omega\);
\item
\( s {}^\smallfrown{} A_m \equiv_{\Lip(d)} A_m \) for every \( m \in \omega \) and \( s \in \pre{< \omega}{\omega} \);
\item
\( A_m \equiv_\W A \) for every \(m \in \omega \).
\end{enumerate}

Recursively define the sequence \( \langle n_k \mid k \in \omega \rangle \) by setting 
\begin{align*}
n_0 & = 0, \\ 
n_{k+1} & = n_k \cdot n_k +1,
\end{align*}%
and for \( i \in \omega \) let \( \# i \) be the unique \( k \in \omega \) such that \( n_k \leq i < n_{k+1} \) (so that, in particular, \( \# n_k = k \)).
Finally, for \( X \subseteq \omega \) set
\[ 
\psi(X) = \bigoplus_{i \in \omega} A_{3 \# i + \rho_X(\# i)},
 \] 
where \( \rho_X \colon \omega \to 2 \) is the characteristic function of the set \( X \) defined by \( \rho_X(j) = 1 \iff j \in X \).
It is trivial to check that e) implies \( \psi(X) \equiv_\W A \) for every \( X \subseteq \omega \). We claim that \( \psi \) is as desired.

First we show that if \( X,Y \subseteq \omega \) are such that \( X \subseteq^* Y \), 
then \( \psi(X) \leq_{\L(d_0)} \psi(Y) \). Fix \( \bar{k} \in \omega \) such that \( \forall k \geq \bar{k} \, (k \in X \Rightarrow k \in Y) \).
For \( k < \bar{k} \), let \( g_k \) be a witness of 
\( A_{3k+\rho_X(k)}  \leq_{\L(d)} A_{3\bar{k}+\rho_Y(\bar{k})} \), which exists by property b) above; for 
\( k \geq \bar{k} \), let \( g_k \) be a witness of \( A_{3k+\rho_X(k)} \leq_{\L(d)} A_{3k+\rho_Y(k)} \), which exists by our 
choice of \( \bar{k} \) and b) again. Then define \( f \colon \pre{\omega}{\omega} \to \pre{\omega}{\omega} \) by setting 
for every \( i \in \omega \) and \( x \in \pre{\omega}{\omega} \)
\[ 
f(i {}^\smallfrown{} x) = \max \{ i, n_{\bar{k}} \} {}^\smallfrown{} g_{\# i}(x).
 \] 
It is straightforward to check that \( f \) reduces \( \psi(X) \) to \( \psi(Y) \), so it remains only to check that \( f \in \L(d_0) \). 
Fix \( x, y \in \pre{\omega}{\omega}\). If \( x(0) = y(0) \), then \( f(x)(0) = f(y)(0) \), so that both \( d_0(x,y) = d(x,y) \) and
 \( d_0(f(x),f(y)) = d(f(x),f(y)) \): therefore the inequality \( d_0(f(x),f(y)) \leq d_0(x,y) \) follows from the fact that all the \( g_k \)'s are in 
\( \L(d) \) together with the observation that the definition of \( f \) on \( \Nbhd_{\langle x(0) \rangle} = \Nbhd_{\langle y(0) \rangle} \) involves only \( g_{\# x(0)} \). Now assume that \( x(0) \neq y(0) \). If at least one of \( x(0) \) and \( y(0) \) is strictly above \( n_{\bar{k}} \), then 
\( f(x)(0) \neq f(y)(0) \), and so by definition of \( f \) and case assumption
\[ 
d_0(f(x),f(y)) = \max \{ f(x)(0),f(y)(0) \} = \max\{x(0),y(0) \} = d_0(x,y).
 \] 
If instead \( x(0),y(0) \leq n_{\bar{k}} \), then \( f(x)(0) = f(y)(0) = n_{\bar{k}} \), so that
\[ 
d_0(f(x),f(y)) \leq \frac{1}{2} < d_0(x,y)
 \] 
because we assumed \( x(0) \neq y(0) \).
Thus, in all cases \( d_0(f(x),f(y)) \leq d_0(x,y) \), and hence we are done.

\medskip

Assume now that \( X,Y \subseteq  \omega \) are such that \( \psi(X) \leq_{\Lip(d_0)} \psi(Y) \), let \( f \colon  \pre{\omega}{\omega} \to \pre{\omega}{\omega} \) be a witness of this, and let \( 0 \neq l \in \omega \) be such that 
\[ 
d_0(f(x),f(y)) \leq 2^l \cdot d_0(x,y) 
\]
for every \( x,y \in \pre{\omega}{\omega} \).

\begin{claim} \label{claim:claim1}
Fix an arbitrary \( i \in \omega \). If there is
\( 2^{l-1} < j \in \omega \) such that 
\( f(\Nbhd_{\langle i \rangle}) \cap \Nbhd_{\langle j \rangle} \neq \emptyset \), then \( f(\Nbhd_{\langle i \rangle}) \subseteq \Nbhd_{\langle j \rangle} \).
\end{claim}

\begin{proof}[Proof of the Claim]
Let \( x \in \pre{\omega}{\omega} \) be such that \( f(i {}^\smallfrown{}  x)(0) = j \), and
suppose towards a contradiction that there is \( y \in \pre{\omega}{\omega} \) such that \( f(i {}^\smallfrown{}  y)(0) \neq j   \). Then
\[ 
d_0(f(i {}^\smallfrown{} x),f(i {}^\smallfrown{}  y)) = \max \{ f(i {}^\smallfrown{}  x)(0),f(i {}^\smallfrown{} y)(0) \} \geq j > 2^{l-1}.
 \] 
But since \( d_0(i {}^\smallfrown{} x,i {}^\smallfrown{} y) \leq \frac{1}{2} \), by our choice of \( l \) we  get
\[ 
d_0(f(i {}^\smallfrown{}  x),f(i {}^\smallfrown{}  y )) \leq 2^l \cdot d_0(i {}^\smallfrown{}  x, i {}^\smallfrown{}  y) \leq 2^l \cdot \frac{1}{2} = 2^{l-1} < d_0(f(i {}^\smallfrown{} x),f(i {}^\smallfrown{}  y)),
 \] 
a contradiction.
\end{proof}

\begin{claim} \label{claim:claim2}
For every \( i \in \omega \), if \( n_{\# i}  > 2^{l-1} \) then there is \( j\in \omega \) such that \( \# j \geq  \# i \) and 
\( f(\Nbhd_{\langle i \rangle}) \subseteq \Nbhd_{\langle j \rangle} \).
\end{claim}

\begin{proof}[Proof of the Claim]
Set \( s = i {}^\smallfrown{}  (l-1) {}^\smallfrown{} 0^{(l-1)} \), so that \( \leng(s) = l + 1 \). Then \( d_0(x,y) = d(x,y) \leq 2^{-(l+1)} \) 
for every \( x,y \in \Nbhd_s \). By our choice of \( l \), it follows that \( d_0(f(x),f(y)) \leq \frac{1}{2} \), and hence 
\( f(x)(0) = f(y)(0) \) by definition of \( d_0 \). This shows that \( f(\Nbhd_s) \subseteq \Nbhd_{\langle j \rangle} \) 
for some \( j \in \omega \). 

Let \( g \colon \pre{\omega}{\omega} \to \pre{\omega}{\omega} \) be defined by
\[ 
g(x) = 
\begin{cases}
f(x) & \text{if } x \in \Nbhd_s \\
(j+1) {}^\smallfrown{} \vec{0} & \text{otherwise}.
\end{cases}%
 \] 
Then \( g \) reduces \( \psi(X) \cap \Nbhd_s \) to \( \psi(Y) \cap \Nbhd_{\langle j \rangle} \) because \( f \) is a reduction of \( \psi(X) \) 
to \( \psi(Y) \); we claim that \( g \in \Lip(d) \). Fix \( x,y \in \pre{\omega}{\omega} \). If \( x ,y \in \Nbhd_s \), then since we showed that
\( f(x)(0) = f(y)(0) \), and moreover \( x(0) = y(0) \) by \( \leng(s) = l+1 > 0 \), we get
\[ 
d(g(x),g(y)) = d(f(x),f(y)) = d_0(f(x),f(y)) \leq 2^l \cdot d_0(x,y) = 2^l \cdot d(x,y).
 \] 
If \( x,y \notin  \Nbhd_s \), then \( g(x) = g(y) \) and hence \( d(g(x),g(y)) = 0 \leq d(x,y) \). If \( x \in \Nbhd_s \) and \( y \notin \Nbhd_s \), then \( d(x,y) \geq 2^{-l} \), and hence
\[ 
d(g(x),g(y)) = 1 = 2^l \cdot 2^{-l} \leq 2^{l} \cdot d(x,y).
 \] 
The case \( x \notin \Nbhd_s \) and \( y \in \Nbhd_s \) is treated similarly. So in all cases \( d(g(x),g(y)) \leq 2^l \cdot d(x,y) \), and hence \( g \in \Lip(d) \).

Since 
\begin{align*} \psi(X) \cap \Nbhd_s & = (i {}^\smallfrown{}  A_{3 \# i + \rho_X(\# i)}) \cap \Nbhd_{i {}^\smallfrown{} (l-1) {}^\smallfrown{} 0^{(l-1)}} \\
& = \Big(i {}^\smallfrown{}  \bigoplus\nolimits_{n \in \omega} 0^{(n)} {}^\smallfrown{} A_{3 \# i + \rho_X(\# i) -1} \Big) \cap \Nbhd_{i {}^\smallfrown{} (l-1) {}^\smallfrown{} 0^{(l-1)}} \\
& = 
s {}^\smallfrown{} A_{3 \# i + \rho_X(\# i) -1} 
\end{align*}  
and 
\( \psi(Y) \cap \Nbhd_{\langle j \rangle}  = j {}^\smallfrown{} A_{3 \#j+\rho_Y(\# j)} \), it follows from d) and the fact that \( g \) witnesses 
\( \psi(X) \cap \Nbhd_s \leq_{\Lip(d)} \psi(Y) \cap \Nbhd_{\langle j \rangle} \) that 
\( A_{3\# i + \rho_X(\# i) -1} \leq_{\Lip(d)} A_{3 \#j+\rho_Y(\# j)} \). By c), 
\[ 
3 \# i + \rho_X(\# i) -1 \leq  3 \# j + \rho_Y(\# j) , 
\]
which implies \( \# i \leq \# j \).

Finally, since \( \# i \leq \# j \) obviously implies \( n_{\#i} \leq j \), we get that \( f(\Nbhd_{\langle i \rangle}) \subseteq \Nbhd_{\langle j \rangle}\) 
by \( f(\Nbhd_s) \subseteq \Nbhd_{\langle j \rangle} \), \( n_{\#i} > 2^{l-1} \), and Claim~\ref{claim:claim1}.
\end{proof}

By Claim~\ref{claim:claim1}, either \( f(\Nbhd_{\langle 0 \rangle}) \subseteq \bigcup_{i \leq 2^{l-1}} \Nbhd_{\langle i \rangle} \), or 
else \( f(\Nbhd_{\langle 0 \rangle}) \subseteq \Nbhd_{\langle j \rangle} \) for some \( j > 2^{l-1} \). Therefore, in both cases there is 
\( \bar{k} \in \omega \) such that \( n_{\bar{k}} \geq 2^l \) (hence, in particular, also \( n_{\bar{k}} > 2^{l-1} \)) and 
\( f(\Nbhd_{\langle 0 \rangle}) \subseteq \bigcup_{i \leq n_{\bar{k}}} \Nbhd_{\langle i \rangle} \): we claim that
 \( \forall k \geq \bar{k} \, (k \in X \Rightarrow k \in Y) \), so that \( X \subseteq^* Y \).

Fix \( k \geq \bar{k} \). Since \( n_k \geq n_{\bar{k}} > 2^{l-1} \) and clearly \( n_{\# n_k} = n_k \), by Claim~\ref{claim:claim2} 
applied to \( i = n_k \) there is \( j \in \omega \) such that \( \# j \geq \# i = k   \) (which implies \( n_k \leq j \)) and
 \( f( \Nbhd_{\langle n_k \rangle}) \subseteq \Nbhd_{\langle j \rangle} \). Assume towards a contradiction that \( j \geq n_{k+1} \). 
Then since \( f(\Nbhd_{\langle 0 \rangle}) \subseteq \bigcup_{i \leq n_{\bar{k}}} \Nbhd_{\langle i \rangle} \) and
 \( j \geq n_{k+1} > n_k \geq n_{\bar{k}} \geq 2^l \), we would get
\[ 
d_0(f(\vec{0}), f(n_k {}^\smallfrown{}  \vec{0})) = j \geq n_{k+1} > n_k \cdot n_k \geq 2^l \cdot n_k = 2^l \cdot d_0(\vec{0}, n_k {}^\smallfrown{}  \vec{0}),
 \] 
contradicting the choice of \( l \). Therefore \( j < n_{k+1} \), and hence \( \# j = k  = \# n_k \). Since
 \( f( \Nbhd_{\langle n_k \rangle}) \subseteq \Nbhd_{\langle j \rangle} \), arguing as in the proof of Claim~\ref{claim:claim2} 
one can  show that \( \psi(X) \cap \Nbhd_{\langle n_k \rangle} \leq_{\Lip(d)} \psi(Y) \cap \Nbhd_{\langle j \rangle} \). 
Since  by d) 
\[ 
 \psi(X) \cap \Nbhd_{\langle n_k \rangle} = n_k {}^\smallfrown{} A_{3 \# n_k + \rho_X(\# n_k)} \equiv_{\Lip(d)} A_{3 \# n_k + \rho_X(\# n_k)} 
\] 
and  
\[  
\psi(Y) \cap \Nbhd_{\langle j \rangle} = j {}^\smallfrown{} A_{3 \# j + \rho_Y(\# j)} \equiv_{\Lip(d)} A_{3 \# j + \rho_Y(\# j)} ,
\]
 this implies that 
\[ 
3k + \rho_X(k) = 3 \# n_k + \rho_X(\# n_k) \leq 3 \# j + \rho_Y(\# j) = 3k + \rho_Y(k)
\]
by c). Therefore \( \rho_X(k) \leq \rho_Y(k) \), and hence we are done.
\end{proof}

\begin{corollary} \label{cor:maind_0}
The partial order \( ( \pow(\omega), \subseteq^*) \) can be embedded into both%
\footnote{The closure of \( \boldsymbol{\Delta}^0_1 \) under both \( \equiv_{\L(d_0)} \) and \( \equiv_{\Lip(d_0)} \) follows from the fact that 
the metric \( d_0 \) is compatible with the topology on \( \pre{\omega}{\omega} \) by Proposition~\ref{prop:d_0compatible}, and hence \( \L(d_0), \Lip(d_0) \subseteq \W \) --- see also Proposition~\ref{prop:inclusiond_0}.}
\( \Deg_{\boldsymbol{\Delta}^0_1}(\L(d_0)) \) and \( \Deg_{\boldsymbol{\Delta}^0_1}(\Lip(d_0)) \). In particular, both degree-structures contain antichains of size \( \pre{\omega}{2} \) (in the sense of Corollary~\ref{cor5}) and infinite descending chains.
\end{corollary}

Using the Parovicenko's result~\cite{Parovivcenko:1963} that (under \( \AC \)) all partial orders of size \( \aleph_1 \) embed into \( ( \pow(\omega), \subseteq^*) \), we get the following corollary.

\begin{corollary}
Assume \( \AC \). Then every partial order of size \( \aleph_1 \) can be embedded into both
 \( \Deg_{\boldsymbol{\Delta}^0_1}(\L(d_0)) \) and \( \Deg_{\boldsymbol{\Delta}^0_1}(\Lip(d_0)) \).
\end{corollary}%

For what concerns the mutual relationships with respect to inclusion of the classes of functions related to \( d_0 \) and \( d \) considered above, we have the following full description.

\begin{proposition} \label{prop:inclusiond_0}
\begin{enumerate}[(1)]
\item
\(  \c(d) \subsetneq \L(d_0) \subsetneq \L(d)  \)
\item
\( \L(d) \not\subseteq \Lip(d_0) \subsetneq \Lip(d) \subsetneq \UCont(d_0) = \UCont(d) \).
\end{enumerate}%
\end{proposition}

\begin{proof}
(1) Fix \( f \in \c(d) \subseteq \L(d) \). Since \( d(f(x),f(y)) < 1 \) for every \( x,y \in \pre{\omega}{\omega} \), it follows from ($\dagger$) that \( d(f(x),f(y)) = d_0(f(x),f(y)) \). Therefore, for every \( x,y \in \pre{\omega}{\omega} \) we have 
\[ 
d_0(f(x),f(y)) = d(f(x),f(y)) \leq d(x,y) \leq d_0(x,y),
 \] 
whence \( f \in \L(d_0) \). 

To show \( \L(d_0) \subseteq \L(d) \), let \( f \in \L(d_0) \) and let \(x,y \in \pre{\omega}{\omega} \). If \( x(0) \neq y(0) \), then \( d(x,y) = 1 \) and trivially \( d(f(x),f(y)) \leq d(x,y) \). If instead \( x(0) = y(0) \), then \( d(x,y) = d_0(x,y) \), whence
\[ 
d(f(x),f(y)) \leq d_0(f(x),f(y)) \leq d_0(x,y) = d(x,y).
 \] 

Finally, both inclusions are proper because the functions \( f =  \id_{\pre{\omega}{\omega}} \) and \( g \colon \pre{\omega}{\omega} \to \pre{\omega}{\omega}\colon n {}^\smallfrown{}  x \mapsto (n+1) {}^\smallfrown{}  x \) are in \( \L(d_0) \setminus \c(d) \) and \(\L(d) \setminus \L(d_0) \), respectively.

(2) Consider the function \( f \colon  \pre{\omega}{\omega} \to \pre{\omega}{\omega} \colon n {}^\smallfrown{} x \mapsto (n^2+1) {}^\smallfrown{} x \). Clearly \( f \in \L(d) \), and since for every \( 0 \neq  n \in \omega \)
\[ 
d_0(f(\vec{0}), f(n {}^\smallfrown{}  \vec{0})) = n^2 + 1 > n \cdot n  = n \cdot d_0(\vec{0}, n {}^\smallfrown{}  \vec{0}),
\] 
we get \( f \notin \Lip(d_0) \). This shows \( \L(d) \not\subseteq \Lip(d_0) \), and hence also \( \Lip(d) \not\subseteq \Lip(d_0) \) by \( \L(d) \subseteq \Lip(d) \).

The inclusion \( \Lip(d_0) \subseteq \Lip(d) \) can be proved similarly to the inclusion \( \L(d_0) \subseteq \L(d) \) above, while \( \UCont(d_0) = \UCont(d) \) follows directly from ($\dagger$). Since clearly \( \Lip(d) \subsetneq \UCont(d) \), we are done.
\end{proof}

\begin{corollary}[\( \AD^\L + \BP \)]
The \( \UCont(d_0) \)-hierarchy is isomorphic to the \( \L(d) \)-hierarchy (see Theorem~\ref{th:Lhierarchy} and Figure~\ref{fig:Lhierarchy}), and in fact it coincides with both the \( \UCont(d) \)- and the \( \Lip(d) \)-hierarchy.
\end{corollary}

Since obviously \( \L(d_0) \subsetneq \Lip(d_0) \), from Proposition~\ref{prop:inclusiond_0} and Corollary~\ref{cor:maind_0} we 
get that both \( \L(d_0) \) and \( \Lip(d_0) \) are examples of  classes of functions which are much larger than the set of 
contractions \( \c = \c(d) \) considered in Section~\ref{sec:contractions}, but which still miss the crucial condition of 
containing \( \L = \L(d) \), and in fact they induce very bad degree-structures.

Slightly modifying the definition of the metric \( d_0 \), we can get a closely related metric \( d_1 \) which is much closer to the standard metric \( d \),
in the sense that in this case \( \Lip(d_1) = \Lip(d) \) (while keeping the conditions \( \L(d_1) \subsetneq \L(d) \) and \( \UCont(d_1) = \UCont(d) \)). 
We will see that also in this case the degree-structure \( \Deg(\L(d_1)) \) is very bad, while the degree-structures \( \Deg(\Lip(d_1)) \) and 
\( \Deg(\UCont(d_1)) \) both coincide with \( \Deg(\Lip(d)) = \Deg(\UCont(d)) \) (and are therefore isomorphic to the classical \( \L(d) \)-hierarchy described in Theorem~\ref{th:Lhierarchy}, see Figure~\ref{fig:Lhierarchy}) when assuming \( \AD^\L +\BP \).

\begin{definition}
Let \( d_1 \colon (\pre{\omega}{\omega})^2 \to \RR^+ \) be the metric on \( \pre{\omega}{\omega} \) defined  by:
\[ 
d_1(x,y) = 
\begin{cases}
0 & \text{if } x = y \\
d(x,y) & \text{if } x(0) = y(0)\\
2 - 2^{-(\max \{ x(0),y(0) \}-1)} & \text{if } x(0) \neq y(0) .
\end{cases}%
 \] 
\end{definition} 

\noindent
So the metric space \( (\pre{\omega}{\omega},d_1) \) is constructed exactly as the space \( (\pre{\omega}{\omega},d_0) \) except for the fact 
that we modify the distances used to ``glue'' together the subspaces \( \Nbhd_{\langle n \rangle} \) of \( (\pre{\omega}{\omega},d) \) in such a way that 
they form a bounded set.

Clearly,  ($\dagger$) remains true also after replacing \( d_0 \) with \( d_1 \), it is still the case that \( d(x,y) \leq d_1(x,y) \) for every \( x,y \in \pre{\omega}{\omega} \), and arguing as in Proposition~\ref{prop:d_0compatible} one sees that
\( d_1 \) is a complete ultrametric compatible with the topology of \( \pre{\omega}{\omega} \). 

\begin{proposition} \label{prop:inclusiond_1}
\begin{enumerate}[(1)]
\item
\(  \L(d_1) = \L(d_0)  \). Therefore \( \c(d) \subsetneq \L(d_1) \subsetneq \L(d) \).
\item
\( \Lip(d_1)  =  \Lip(d) \) and \( \UCont(d_1) = \UCont(d) \).
\end{enumerate}%
\end{proposition}

\begin{proof}
(1) Use the fact that the map \( i \) defined by \( i(0) = 0 \), \( i(2^{-(n+1)}) = 2^{-(n+1)} \), and \( i(n+1) = 2 - 2^{-n} \) 
is an order-preserving map such that \( d_1(x,y) = i(d_0(x,y)) \) for every \(x,y \in \pre{\omega}{\omega} \).

(2) The equality \( \UCont(d_1) = \UCont(d) \) follows again from the analogues of ($\dagger$) with \( d_0 \) replaced by \( d_1 \), so it remains only 
to show that \( \Lip(d_1)  =  \Lip(d) \). 

Assume first that \( f \in \Lip(d) \) and let \( l \in \omega \) be such that \( d(f(x),f(y)) \leq 2^l \cdot d(x,y) \) for every 
\( x,y \in \pre{\omega}{\omega} \). We consider two cases: if \( x \restriction (l+1) = y \restriction (l+1) \), then \( d(x,y) \leq 2^{-(l+1)} \), 
and hence \( d(f(x),f(y)) \leq 2^l \cdot 2^{-(l+1)} = \frac{1}{2} \). By the analogous of ($\dagger$) for \( d_1 \) and the definition of \( d_1 \), this means that both 
\( d_1(x,y) = d(x,y) \) and \( d_1(f(x),f(y)) = d(f(x),f(y)) \), whence \( d_1(f(x),f(y))  \leq 2^l \cdot d_1(x,y) \). If instead
 \( x \restriction (l+1) \neq y \restriction (l+1) \), then \( d_1(x,y) \geq 2^{-l} \), and since all distances realized by \( d_1 \) are
 bounded by \(  2 \) we get that 
\[ 
d_1(f(x),f(y)) \leq 2  = 2^{l+1} \cdot 2^{-l} \leq 2^{l+1} \cdot d_1(x,y) .
\] 
Therefore 
\( d_1(f(x),f(y)) \leq 2^{l+1} \cdot d_1(x,y) \) for every \( x,y \in \pre{\omega}{\omega} \), and hence \( f \in \Lip(d_1) \). 
This shows \( \Lip(d) \subseteq \Lip(d_1) \). The inclusion \( \Lip(d_1) \subseteq \Lip(d) \) can be proved in a similar way (or, alternatively, using the argument contained in the proof of \( \Lip(d_0) \subseteq \Lip(d) \) in Proposition~\ref{prop:inclusiond_0}), hence we are done.
\end{proof}

From Proposition~\ref{prop:inclusiond_1} and Theorem~\ref{th:maind_0} we immediately get the following result.

\begin{theorem} \label{th:maind_1}
Let \( A \subseteq \pre{\omega}{\omega} \) be a \( \W \)-selfdual set. Then there is a map \( \psi \colon \pow(\omega) \to [A]_\W \) such that for every \( X,Y \subseteq \omega \)
\[ 
X \subseteq^* Y \iff \psi(X) \leq_{\L(d_1)} \psi(Y).
 \] 
In particular, \( ( \pow(\omega), \subseteq^*) \) embeds into
\( \Deg_{\boldsymbol{\Delta}^0_1}(\L(d_1)) \), and  hence \( \Deg_{\boldsymbol{\Delta}^0_1}(\L(d_1)) \) contains both antichains of size \( \pre{\omega}{2} \) (in the sense of Corollary~\ref{cor5}) and infinite descending chains.
\end{theorem}

\section{Questions and open problems} \label{sec:questions} 

By Theorem~\ref{th:Fhierarchy}, the inclusion \( \F \supseteq \L \) is a sufficient condition for \( \Deg(\F) \) being very good (under suitable determinacy assumptions).
 However, literally this is not a necessary condition: in fact, letting \( \F = \{ f \colon \pre{\omega}{\omega} \to \pre{\omega}{\omega} \mid f \text{ is two-valued} \} \cup \{ \id \} \), one gets that \( \F \not\supseteq \L \), but \( \Deg(\F) \) consists of the 
\( \F \)-nonselfdual pair \( \{ [\pre{\omega}{\omega}]_\F, [\emptyset]_\F \} \) plus a unique \( \F \)-degree above it containing all
 sets \(\emptyset,\pre{\omega}{\omega} \neq A \subseteq \pre{\omega}{\omega} \), and is thus (trivially) very good. This phenomenon is apparently due to the fact that such an \( \F \) is too 
large (that is, it is in bijection with \( \pow(\pre{\omega}{\omega}) \)), and this makes its induced degree-structure collapse to an extremely simple finite structure. In contrast, to the best of our knowledge the problem of 
whether any degree-structure induced by a \emph{not too large} \( \F \not\supseteq \L \) must be (very) bad remains open. 

\begin{question}
Work in \( \ZF + \DC(\RR) \) (or in \( \ZF + \DC(\RR) + \AD \)), and 
let \( \F \) be a collection of functions from \( \pre{\omega}{\omega} \) into itself closed under composition and containing \( \id \). Assume that \( \F \) is a surjective image of \( \pre{\omega}{\omega} \).
 Is it true that if \( \F \not\supseteq \L \) then \( \Deg(\F) \) is (very) bad?
\end{question}%

Another open problem related to the results in Section~\ref{sec:changingmetric} is the following.

\begin{question}
Is there a complete ultrametric \( d' \) on the Baire space which is compatible with its topology and such that \( \L(d) \not\subseteq \L(d'),\Lip(d'),\UCont(d') \)? Can \( d' \) be chosen so that all the hierarchies \( \Deg(\L(d')) \), \(\Deg(\Lip(d')) \), and \( \Deg(\UCont(d')) \) are (very) bad?
\end{question}

All the degree-structures considered in this paper were either very good, or else (very) bad. Thus it seems natural to ask the following:

\begin{question}
Is there any ``natural'' collection of functions from \( \pre{\omega}{\omega} \) into itself (closed under composition and containing \( \id \)) such that \( \Deg(\F) \) is good but not very good?
\end{question}%

Finally, it could be interesting to further investigate the notion of selfcontractible subsets of  metric spaces considered in Sections~\ref{sec:definitions} and~\ref{sec:contractions}.

\begin{question}
Given  \( 1 \leq \xi < \omega_1 \), is it true that every proper \( \boldsymbol{\Sigma}^0_\xi \) or proper \( \boldsymbol{\Pi}^0_\xi \) subset of \( \RR \) is selfcontractible? What if \( \RR \) is replaced by an arbitrary uncountable Polish space? Is it possible to characterize the collection of all selfcontractible subsets of \( \RR \) (or, more generally, of an arbitrary uncountable Polish space) similarly to Corollaries~\ref{cor2} and~\ref{cor:2'}?
\end{question}


\end{document}